\documentclass[10pt]{article}
\advance\oddsidemargin by -1.0cm
\advance\evensidemargin by -1.0cm
\advance\topmargin by -1.0cm
\usepackage{latexsym,amssymb,yfonts} %
\usepackage{amsmath}
\usepackage{amsthm}
\usepackage{mathtools}
\usepackage{enumerate}
\usepackage{psfrag}
\usepackage{graphicx}
\usepackage[colorlinks=true,linkcolor=black,urlcolor=black,citecolor=black,bookmarks]{hyperref}
\usepackage{multirow}
\usepackage{booktabs}
\usepackage{arydshln}
\numberwithin{equation}{section}
\newtheorem{theorem}{Theorem}[subsection]
\newtheorem{prop}{Proposition}[subsection]

\newtheorem{cor}{Corollary}[subsection]

\newtheorem{lem}{Lemma}[subsection]

\newtheorem{lemma}{Lemma}[subsection]

\theoremstyle{definition}
\newtheorem{df}{Definition}[subsection]

\newtheorem{definition}{Definition}[subsection]

\newtheorem{ex}{Example}[subsection]

\theoremstyle{remark}
\newtheorem{rem}{Remark}[subsection]

\newtheorem*{notation}{Notation}
\newtheorem*{remark*}{Remark}

\numberwithin{table}{section}
\newcommand{\bdm}{\begin{displaymath}}
\newcommand{\edm}{\end{displaymath}}
\newcommand{\be}{\begin{equation}}
\newcommand{\ee}{\end{equation}}

\newcommand{\cyclic}[1]{\stackrel{\scriptsize #1}{\mathfrak{S}}}
\newcommand{\R}{\mathbb{R}}

\renewcommand{\H}{\mathcal{H}}
\newcommand{\V}{\mathcal{V}}

\newcommand{\so}{\ensuremath{\mathfrak{so}}}
\newcommand{\SO}{\ensuremath{\mathrm{SO}}}

\newcommand{\spin}{\ensuremath{\mathfrak{spin}}}
\newcommand{\ad}{\mathrm{ad}\,}
\newcommand{\dd}{\mathrm{d}}
\newcommand\extalg{%
  \newlength{\len}%
  \settoheight{\len}{V}%
  \mathbin{%
    \resizebox{0.93\len}{0.93\len}{$\wedge$}%
    \kern-0.1em%
  }}%
\newcommand{\intprod}{\mathbin{\hbox to 0.7ex{%
      \kern-0.3ex
      \vrule height0.0777ex width0.971ex depth0ex
      \kern-0.055ex
      \vrule height1.165ex width0.0777ex depth0ex\hss}}%
}%

\newcommand{\N}{\mathbb{N}}
\newcommand{\Cl}{C\!\ell}
\newcommand{\raa}[1]{\renewcommand{\arraystretch}{#1}}
\begin{document}

\title{Homogeneous non-degenerate  $3$-$(\alpha,\delta)$-Sasaki manifolds
and submersions over quaternionic K\"ahler spaces}

\author{Ilka Agricola, Giulia Dileo, and Leander Stecker}
\date{}

\maketitle

\begin{abstract}
We show that every $3$-$(\alpha,\delta)$-Sasaki manifold of dimension $4n + 3$ admits a locally defined
Riemannian submersion over a quaternionic K\"ahler manifold of scalar curvature $16n(n+2)\alpha\delta$.
In the non-degenerate case we describe all homogeneous $3$-$(\alpha,\delta)$-Sasaki manifolds
fibering over symmetric Wolf spaces and over their non-compact dual symmetric
spaces. If $\alpha\delta> 0$, this yields a complete classification of homogeneous
$3$-$(\alpha,\delta)$-Sasaki manifolds. For $\alpha\delta< 0$, we  provide a general construction of
homogeneous  $3$-$(\alpha,\delta)$-Sasaki manifolds fibering over non-symmetric Alekseevsky spaces, the lowest
possible dimension of such a manifold being $19$.
\end{abstract}


\tableofcontents


\pagestyle{headings}


\phantom{x}

\vspace{1cm}

\medskip
\noindent
{\small

}


\section{Introduction and basic notions}
%
\subsection{Introduction}
%
Sasaki manifolds have been studied since the 1970s as an odd dimensional counterpart to K\"ahler geometry. Similarly, $3$-Sasaki manifolds are considered the $(4n+3)$-dimensional analogue to hyper-K\"ahler (hK) geometry. However, while these geometries are linked via the hK cone of a $3$-Sasaki manifold, $3$-Sasaki geometry also connects to another $4n$-dimensional geometry, namely quaternionic K\"ahler (qK) manifolds. Initially shown in the regular case by Ishihara and in full generality by C.~Boyer, K.~Galicki and B.~Mann in '94, every $3$-Sasaki manifold locally admits a fibration over a qK orbifold \cite{BGM94}. This led to the classification of all homogeneous $3$-Sasaki manifolds. The reverse construction is given by taking the Konishi bundle of a positive scalar curvature qK space, i.e.\ the orthonormal frame bundle of the quaternionic structure \cite{konishi}. For qK manifolds with negative scalar curvature one does not obtain a $3$-Sasaki manifold but a so-called pseudo $3$-Sasaki structure \cite{tanno96}. This notion, however, did not gather as much traction since it comes with a metric of semi-Riemannian signature $(4n, 3)$.

More recently the first two authors investigated Riemannian almost $3$-contact metric manifolds by means of connections with torsion \cite{AgrDil}. They found necessary and sufficient conditions for the existence of compatible connections. Along their investigations, they discovered the more specific class of $3$-$(\alpha,\delta)$-Sasaki manifolds connecting many examples on which partial results were known previously. In particular, they showed that pseudo $3$-Sasaki structures can be turned into \emph{negative}  $3$-$(\alpha,\delta)$-Sasaki manifolds (i.\,e.~with $\alpha\delta<0$).

This paper aims to connect both worlds and presents $3$-$(\alpha,\delta)$-Sasaki geometry as the go-to structure above any qK space. We quickly review all necessary notions involving $3$-$(\alpha,\delta)$-Sasaki structures in Section 1. Using results by R.~Cleyton, A.~Moroianu and U.~Semmelmann \cite{CleyMorSemm} we obtain a locally defined Riemannian submersion over a qK space establishing the canonical connection as the link between both geometries. This is done in Section 2. In the $3$-Sasaki case we recover the result of Boyer, Galicki, Mann.
We further show that the scalar curvature on the base is a positive multiple of $\alpha\delta$. Thus, for negative and degenerate $3$-$(\alpha,\delta)$-Sasaki manifolds we obtain submersions onto qK spaces of negative scalar curvature, respectively hK spaces. This suggests to investigate non-degenerate homogeneous $3$-$(\alpha,\delta)$-Sasaki manifolds by looking at homogeneous qK manifolds of non-vanishing scalar curvature. Section 3 is therefore devoted to a hands on construction of homogeneous $3$-$(\alpha,\delta)$-Sasaki spaces over all known homogeneous qK manifolds. This yields a construction over symmetric Wolf spaces, deforming the description given in \cite{Draperetall}
(see also \cite[Theorem 4]{Bielawski}), and by similar means their non-compact duals. Additionally, homogeneous $3$-$(\alpha,\delta)$-Sasaki manifolds over Alekseevsky spaces are constructed using a description of the latter given by V.\ Cort\'es in \cite{CortesNewConstr}. We provide detailed descriptions of the $7$-dimensional Aloff-Wallach space, its negative counterpart fibering over the $4$-dimensional Wolf space $\mathrm{SU}(3)/S(\mathrm{U}(2)\times\mathrm{U}(1))$, respectively its non-compact dual, as well as of the homogeneous $3$-$(\alpha,\delta)$-Sasaki space $\hat{\mathcal{T}}(1)$ in dimension $19$ sitting above the non-symmetric Alekseevsky space $\mathcal{T}(1)$. In Section 4 we compute the Nomizu map associated to the canonical connection, the necessary tool for any further investigation of these spaces. In the symmetric base case we find the Nomizu map of the Levi-Civita connection as well.
%

\subsection{Review of $3$-$(\alpha,\delta)$-Sasaki manifolds and their basic properties}
%
We review some basic definitions and properties on almost contact metric manifolds.
This serves mainly as a  reference.

An \emph{almost contact metric structure} on a $(2n+1)$-dimensional differentiable manifold
$M$ is a quadruple $(\varphi,\xi,\eta,g)$, where $\varphi$ is a $(1,1)$-tensor
field, $\xi$ a vector field, $\eta$ a $1$-form, $g$ a Riemannian metric, such that
\begin{gather*} 
\varphi^2=-I+\eta\otimes \xi,\quad  \eta(\xi)=1,\quad \varphi (\xi) =0,\quad \eta \circ \varphi =0,\\
g(\varphi X,\varphi Y)=g(X,Y)-\eta (X) \eta(Y)\quad \forall X,Y\in{\frak X}(M).
\end{gather*}
It follows that $\varphi$ has rank $2n$ and the tangent bundle of $M$ splits as
$TM=\mathcal{H}\oplus\langle\xi\rangle$, where $\H$ is the
$2n$-dimensional distribution defined by
$\H=\mathrm{Im}(\varphi)=\ker\eta=\langle \xi\rangle^\perp$. In particular, $\eta=g(\cdot,\xi)$.
The vector field $\xi$ is called the \emph{characteristic} or \emph{Reeb vector field}.
The almost contact metric structure is said to be
\emph{normal} if $ N_\varphi\coloneqq[\varphi,\varphi]+d\eta\otimes\xi$ vanishes,
where $[\varphi,\varphi]$ is the Nijenhuis torsion of $\varphi$ \cite{BLAIR}.

An \emph{$\alpha$-Sasaki manifold} is defined as a normal almost contact metric manifold
such that
$
d\eta\, =\, 2\alpha\Phi$, $\alpha\in\R^*,
$
where $\Phi$ is the fundamental $2$-form defined by
$\Phi(X,Y)=g(X,\varphi Y)$. For $\alpha=1$, this is a \emph{Sasaki manifold}. The $1$-form
$\eta$ of an $\alpha$-Sasaki structure is a \emph{contact form}, in the sense that
$\eta\wedge (d\eta)^n\ne 0$ everywhere on $M$. The Reeb vector field is always Killing.

An  \emph{almost $3$-contact metric manifold}  is a differentiable  manifold $M$ of dimension
$4n+3$ endowed with three almost contact metric structures $(\varphi_i,\xi_i,\eta_i,g)$,
$i=1,2,3$, sharing the same Riemannian metric $g$, and satisfying the following compatibility
relations
%
\bdm
\varphi_k=\varphi_i\varphi_j-\eta_j\otimes\xi_i=-\varphi_j\varphi_i+\eta_i\otimes\xi_j,\quad
\xi_k=\varphi_i\xi_j=-\varphi_j\xi_i, \quad
\eta_k=\eta_i\circ\varphi_j=-\eta_j\circ\varphi_i
\edm
%
for any even permutation $(ijk)$ of $(123)$ \cite{BLAIR}.
The tangent bundle of $M$ splits into the orthogonal sum $TM=\H\oplus\V$, where $\H$ and
$\V$ are respectively the \emph{horizontal} and the \emph{vertical} distribution, defined by
\[
\H \, \coloneqq\, \bigcap_{i=1}^{3}\ker\eta_i,\qquad
\V\, \coloneqq\, \langle\xi_1,\xi_2,\xi_3\rangle.
\]
In particular $\H$ has rank $4n$ and the three Reeb vector
fields $\xi_1,\xi_2,\xi_3$ are orthonormal.
The manifold is said to  be \emph{hypernormal} if each  almost contact metric structure
$(\varphi_i,\xi_i,\eta_i,g)$ is normal.
We denote an almost $3$-contact metric manifold by $(M,\varphi_i,\xi_i,\eta_i, g)$, understanding
that the index is running from $1$ to $3$.

One of the most interesting classes of almost $3$-contact metric manifolds is given by
\emph{$3$-$\alpha$-Sasaki manifolds}, for which each of the three structures is $\alpha$-Sasaki. For
$\alpha=1$, this is just the definition of a \emph{$3$-Sasaki manifold}. As a  comprehensive introduction
to Sasaki and $3$-Sasaki geometry, we refer to \cite{Boyer&Galicki}.
In the recent paper \cite{AgrDil} the new class of $3$-$(\alpha,\delta)$-Sasaki manifolds
was introduced, generalizing $3$-$\alpha$-Sasaki manifolds.
\begin{df}\label{df.3ad-Sasaki}
An almost $3$-contact metric manifold $(M,\varphi_i,\xi_i,\eta_i,g)$ is called a
\emph{$3$-$(\alpha,\delta)$-Sasaki manifold} if it satisfies
\begin{equation}\label{differential_eta}
d\eta_i=2\alpha\Phi_i+2(\alpha-\delta)\eta_j\wedge\eta_k
\end{equation}
for every even permutation $(ijk)$ of $(123)$, where $\alpha\neq 0$ and
$\delta$ are real constants.
A $3$-$(\alpha,\delta)$-Sasaki manifold is called \emph{degenerate} if
$\delta=0$ and \emph{non-degenerate} otherwise.
Non-degenerate $3$-$(\alpha,\delta)$-Sasaki manifolds will be distinguished into
\emph{positive} and \emph{negative} ones, depending on whether  $\alpha\delta>0$ or $\alpha\delta<0$.
\end{df}
\begin{rem}
Recall that the distinction into degenerate, positive, and negative
$3$-$(\alpha,\delta)$-Sasaki manifolds stems from their behaviour under
 \emph{$\H$-homothetic deformations} \cite[Section 2.3]{AgrDil}:
%
\bdm
\eta_i'=c\eta_i,\quad \xi_i'=\frac{1}{c}\xi_i,\quad \varphi_i'=\varphi_i,\quad
g'=ag+b\sum_{i=1}^3\eta_i\otimes\eta_i \qquad a>0,\, c^2=a+b>0.
\edm
The deformed structure $(\varphi',\xi_i',\eta',g')$ turns out to be $3$-$(\alpha',\delta')$-Sasaki with
$\alpha'=\alpha c/a$, $\delta'=\delta/c$.
In particular, $\H$-homothetic deformations preserve the class of degenerate
$3$-$(\alpha,\delta)$-Sasaki manifolds. In the non-degenerate case the sign of the
product $\alpha\delta$ is also preserved, which justifies the distinction between the positive and
negative case stated in the definition above. In fact a $3$-$(\alpha,\delta)$-Sasaki manifold
is positive if and only if it is $\H$-homothetic to a $3$-Sasaki manifold, and
negative if and only if it is $\H$-homothetic to a $3$-$(\tilde\alpha,\tilde\delta)$-Sasaki
manifold with $\tilde\alpha=-\tilde\delta=1$.
\end{rem}
We recall some basic properties of $3$-$(\alpha,\delta)$-Sasaki manifolds whose proofs can be found in
\cite{AgrDil}. Any $3$-$(\alpha,\delta)$-Sasaki manifold is shown to be hypernormal,
thus generalizing Kashiwada's theorem \cite{kashiwada}. Hence, for $\alpha=\delta$ one has a
$3$-$\alpha$-Sasaki manifold. Each Reeb vector field $\xi_i$ is Killing and it is an infinitesimal
automorphism of the horizontal distribution $\H$, i.e. $d\eta_i(X,\xi_j)=0$ for every
$X\in\H $ and $i,j=1,2,3$. The vertical distribution $\V$ is integrable with totally
geodesic leaves.
In particular, the commutators of the Reeb vector fields are purely vertical and for every even permutation $(ijk)$ of $(123)$ they are given by
\bdm
[\xi_i,\xi_j]=2\delta\xi_k.
\edm
Meanwhile, the vertical part of commutators of horizontal vector fields is encoded by the fundamental form,
as is shown in the following useful lemma:
\begin{lemma}\label{VComm}
For two horizontal vectors $X,Y\in\mathcal{H}$ we have
\[
[X,Y]_\mathcal{V}=-2\alpha\sum_{i=1}^3\Phi_i(X,Y)\xi_i.
\]
\end{lemma}
\begin{proof}
Since the vertical distribution is spanned by the Reeb vector fields, we have
\[
[X,Y]_{\mathcal V}=\sum_{i=1}^3\eta_i([X,Y])\xi_i=-\sum_{i=1}^3d\eta_i(X,Y)\xi_i=-2\alpha\sum_{i=1}^3\Phi_i(X,Y)\xi_i.\qedhere
\]
\end{proof}
By the same argument $[X,Y]_\mathcal{V}=0$ if $X\in\mathcal{H}$ and $Y=\xi_j$, $j=1,2,3$, which is equivalent to the fact that $\mathrm{d}\eta_i(X,\xi_j)=0$, $i=1,2,3$. 

A remarkable property of $3$-$(\alpha,\delta)$-Sasaki manifolds is that they are
\emph{canonical almost $3$-contact metric manifolds}, in the sense of \cite{AgrDil}, which is equivalent
to the existence of a \emph{canonical connection}.

We recall here some basic facts about connections with totally skew-symmetric torsion---we
refer to \cite{Ag} for further details. A metric
connection $\nabla$ with torsion $T$ on a Riemannian manifold $(M,g)$ is said to have
\emph{totally skew-symmetric torsion},
or \emph{skew torsion} for short, if the $(0,3)$-tensor field $T$ defined by
\[T(X,Y,Z)=g(T(X,Y),Z)\]
is a $3$-form. The relation between $\nabla$ and the Levi-Civita connection
$\nabla^g$ is then given by
\bdm
\nabla_XY=\nabla^g_XY+\frac{1}{2}T(X,Y).
\edm
It is well-known that any Sasaki manifold $(M,\varphi,\xi,\eta,g)$ admits a \emph{characteristic connection},
i.\,e.~a unique metric connection $\nabla$ with skew torsion such that $\nabla\eta=\nabla\varphi=0$. Its torsion
is given by $T=\eta\wedge d\eta$ \cite{FrIv}. As a
consequence, a $3$-Sasaki manifold $(M,\varphi_i,\xi_i,\eta_i,g)$ cannot admit any metric connection
with skew torsion such that $\nabla\eta_i=\nabla\varphi_i=0$ for every $i=1,2,3$. By relaxing
the requirement on the parallelism of the structure tensor fields in a suitable way,
one can define a large class of almost $3$-contact metric
manifolds, called \emph{canonical}, including $3$-$(\alpha,\delta)$-Sasaki manifolds, and thus
$3$-Sasaki manifolds.

Any $3$-$(\alpha,\delta)$-Sasaki manifold $(M,\varphi_i,\xi_i,\eta_i,g)$ is canonical, in
the sense that it admits a unique metric connection $\nabla$ with skew torsion such that
\begin{equation}\label{canonical}
\nabla_X\varphi_i\, =\, \beta(\eta_k(X)\varphi _j -\eta_j(X)\varphi _k) \quad
\forall X\in{\frak X}(M)
\end{equation}
for every even permutation $(ijk)$ of $(123)$, where $\beta=2(\delta-2\alpha)$. The
covariant derivatives  of the other structure tensor fields are given by
\[\nabla_X\xi_i=\beta(\eta_k(X)\xi_j-\eta_j(X)\xi_k),\qquad\nabla_X\eta_i=\beta(\eta_k(X)\eta_j-\eta_j(X)\eta_k).\]
If $\delta=2\alpha$, then $\beta=0$ and the canonical connection parallelizes all the structure tensor fields.
Any $3$-$(\alpha,\delta)$-Sasaki manifold with $\delta=2\alpha$ is called \emph{parallel}. Notice that this is
a positive $3$-$(\alpha,\delta)$-Sasaki manifold.

The torsion $T$ of the canonical connection is given by
\begin{equation}\label{torsion01}
T\ =\ 2\alpha\sum_{i=1}^3\eta_i\wedge\Phi_i-2(\alpha-\delta)\eta_{123}\ =\
2\alpha \sum_{i=1}^3\eta_i\wedge \Phi^{\mathcal H}_i+2(\delta-4 \alpha)\,\eta_{123},
\end{equation}
where $\Phi^{\mathcal H}_i=\Phi_i+\eta_{jk}\in\Lambda^2({\mathcal H})$ is the horizontal part of the
fundamental $2$-form $\Phi_i$. Here we put $\eta_{jk}\coloneqq\eta_j\wedge\eta_k$ and $\eta_{123}\coloneqq\eta_1\wedge\eta_2\wedge\eta_3$.
In particular, for every $X,Y\in\frak{X}(M)$,
\begin{equation}\label{torsion02}
T(X,Y)=2\alpha\sum_{i=1}^3\{\eta_i(Y)\varphi_iX-\eta_i(X)\varphi_iY+\Phi_i(X,Y)\xi_i\}-2(\alpha-\delta)
\cyclic{i,j,k}\eta_{ij}(X,Y)\xi_k.
\end{equation}
The symbol $\cyclic{i,j,k}$ means the sum over all even permutations of
$(123)$.
The torsion of the canonical connection satisfies $\nabla T=0$.
The curvature properties of $3$-$(\alpha,\delta)$-Sasaki manifolds will be discussed in detail in a separate
publication \cite{ADS21}. We cite from there
without proof the following special result that will be needed in the following
section. It is a side result of a lengthy and non-trivial, but otherwise straightforward computation.
\begin{prop}[\cite{ADS21}]\label{curv-ids}
The curvature tensor $R$ of the canonical connection of a $3$-$(\alpha,\delta)$-Sasaki manifold
satisfies  for any $X,Y,Z\in\H $ and $i,j,k,l=1,2,3$ the identities
\begin{align}
R(X,\xi_i,Y,\xi_j)&=R(X,Y,Z,\xi_i)=R(\xi_i,\xi_j,\xi_k,X)=0,\label{0Curv}\\
R(\xi_i,\xi_j,\xi_k,\xi_l)&=-4\alpha\beta (\delta_{ik}\delta_{jl}-\delta_{il}\delta_{jk}),\label{1Curv}\\
R(\xi_i,\xi_j,X,Y)&=2\alpha\beta \Phi_k(X,Y),\label{offdiag}\\
R(X,Y,Z,\varphi_iZ)+R(X,Y,\varphi_jZ,\varphi_kZ)&=2\alpha\beta\Phi_i(X,Y)\|Z\|^2,\label{2Curv}
\end{align}
where in the last two identities $(ijk)$ is an even permutation of $(123)$.
\end{prop}
%
%
%
%
%
%
%
%

\section{The Riemannian submersion over a quaternionic K\"ahler base}
%
\subsection{The canonical submersion}
%
In \cite{CleyMorSemm} the authors discuss the geometry of Riemannian manifolds admitting metric connections $\nabla^{\tau}$ with parallel skew torsion $\tau$ and reducible holonomy. This applies, in particular, to the canonical connection of $3$-$(\alpha,\delta)$-Sasaki manifolds. We shortly recall their notation.

Suppose the tangent space $TM$ decomposes under the action of the holonomy group $\mathrm{Hol}$ of $\nabla^\tau$ into a sum of irreducible representations $\mathfrak{v}_1,\dots,\mathfrak{v}_r,\mathfrak{h}_1,\dots,\mathfrak{h}_s$. Here an irreducible submodule is called vertical, adequately denoted by $\mathfrak{v}_j$, if the subspace of $\mathfrak{hol}$ acting purely on $\mathfrak{v}_j$ is trivial. Conversely, a subspace $\mathfrak{h}_a$ is called horizontal if the subspace $\mathfrak{k}_a=\mathfrak{so}(\mathfrak{h}_a)\cap\mathfrak{hol}\neq \{0\}$ of $\mathfrak{hol}$ acting purely on $\mathfrak{h}_a$ is non-trivial.

We need a slight generalization of the results obtained in \cite{CleyMorSemm}.
Suppose the tangent space decomposes into $TM=\mathfrak{v}_1\oplus\dots\oplus\mathfrak{v}_r\oplus\mathfrak{h}_1\oplus\dots\oplus\mathfrak{h}_s$ as before. Let $TM=\mathcal{V}_\Gamma\oplus\mathcal{H}_\Gamma$ be a decomposition such that
\begin{equation}\label{tangentdecomposition}
\mathcal{H}_{\Gamma}\coloneqq\bigoplus_{a=1}^s\mathfrak{h}_a\oplus\bigoplus_{j\in\Gamma_0\setminus\Gamma}\mathfrak{v}_j,
\qquad\qquad \mathcal{V}_{\Gamma}\coloneqq\bigoplus_{j\in\Gamma}\mathfrak{v}_j,
\end{equation}
for some subset $\Gamma\subset\Gamma_0=\{1,\dots,r\}$.
Suppose further that for this decomposition the projection of $\tau$ onto the space $\mathcal{H}_\Gamma\otimes\Lambda^2\mathcal{V}_\Gamma$ satisfies
\begin{equation}\label{projecttau}
0=\mathrm{pr}_{\mathcal{H}_\Gamma\otimes\Lambda^2\mathcal{V}_\Gamma}\tau\in\mathcal{H}_\Gamma\otimes\Lambda^2\mathcal{V}_\Gamma\subset\Lambda^3(\mathcal{H}_\Gamma\oplus\mathcal{V}_\Gamma).
\end{equation}
This condition turns out to be sufficient to prove Lemma 3.7-3.10 and Remark 3.11 from \cite{CleyMorSemm}. We obtain

\begin{cor}\label{CMSThm}
Suppose the decomposition $TM=\mathcal{V}_\Gamma\oplus\mathcal{H}_\Gamma$ from \eqref{tangentdecomposition} fulfills condition \eqref{projecttau}. Then
\begin{enumerate}[a)]
\item the distribution $\mathcal{V}_\Gamma$ is the vertical distribution of a locally defined Riemannian submersion $(M,g)\overset{\pi}{\longrightarrow} (N,g_N)$ with totally geodesic leaves,
\item there exists a $3$-form $\sigma\in\Lambda^3N$ satisfying $\pi^*\sigma=\mathrm{pr}_{\Lambda^3\mathcal{H}_\Gamma}\tau$,
\item $\nabla^\sigma\coloneqq\nabla^{g_N}+\frac 12\sigma$ defines a connection with parallel skew torsion $\sigma$ on $N$. In particular, we have
\begin{equation}\label{pinabla}
\nabla^\sigma_XY=\pi_*(\nabla^\tau_{\overline{X}}\overline{Y}),
\end{equation}
for the horizontal lifts $\overline{X},\overline{Y}\in TM$ of the vectors fields $X,Y\in TN$.
\end{enumerate}
\end{cor}

Equation \eqref{pinabla} is not stated explicitely in \cite{CleyMorSemm} but follows directly from $\nabla^{g_N}_{X}{Y}=\pi_*(\nabla^{g}_{\overline{X}}\overline{Y})$ for Riemannian submersions \cite[Prop. 13]{Petersen}.
To a Riemannian submersion one assigns the O'Neill tensors
\[
\mathcal{A}_XY=(\nabla^g_{X_{\mathcal{H}}}Y_{\mathcal{H}})_{\mathcal{V}}+(\nabla^g_{X_{\mathcal{H}}}Y_{\mathcal{V}})_{\mathcal{H}},
\qquad
\mathcal{T}_XY=(\nabla^g_{X_{\mathcal{V}}}Y_{\mathcal{H}})_{\mathcal{V}}+(\nabla^g_{X_{\mathcal{V}}}Y_{\mathcal{V}})_{\mathcal{H}}.
\]
Here the subscripts denote projection on the respective subspaces. For the submersion above $\mathcal{A}$ and $\mathcal{T}$ simplify:
\begin{lemma}\label{ATensor}
The O'Neill tensors $\mathcal{A}$ and $\mathcal{T}$ associated to the submersion defined by $TM=\mathcal{V}_\Gamma\oplus\mathcal{H}_\Gamma$ are given by
\[
g(\mathcal{A}_XY,Z)=-\frac{1}{2}(\tau(X_{\mathcal{H}_\Gamma},Y_{\mathcal{H}_\Gamma},Z_{\mathcal{V}_\Gamma})+\tau(X_{\mathcal{H}_\Gamma},Y_{\mathcal{V}_\Gamma},Z_{\mathcal{H}_\Gamma})),\qquad \mathcal{T}=0.
\]
\end{lemma}
\begin{proof}
Since $\mathcal{H}_\Gamma$ and $\mathcal{V}_\Gamma$ are $\nabla^\tau$-holonomy invariant
$(\nabla^\tau_XY_{\mathcal{H}_\Gamma})_{\mathcal{V}_\Gamma}=(\nabla^\tau_XY_{\mathcal{V}_\Gamma})_{\mathcal{H}_\Gamma}=0$. Thus, $g(\nabla^g_{X}Y_{\mathcal{H}_\Gamma},Z_{\mathcal{V}_\Gamma})=-\frac 12 \tau(X,Y_{\mathcal{H}_\Gamma},Z_{\mathcal{V}_\Gamma})$ and $g(\nabla^g_{X}Y_{\mathcal{V}_\Gamma},Z_{\mathcal{H}_\Gamma})=-\frac 12 \tau(X,Y_{\mathcal{V}_\Gamma},Z_{\mathcal{H}_\Gamma})$. The first expression follows directly.
The identity $\mathcal{T}=0$ is then an immediate consequence of condition \eqref{projecttau}.
\end{proof}
The vanishing of $\mathcal{T}$ does not come as a surprise since it is equivalent to the fibers being totally geodesic.

\medskip
We now discuss the situation for $3$-$(\alpha,\delta)$-Sasaki manifolds. By \eqref{canonical} the holonomy representation of the canonical connection $\nabla$ of a $3$-$(\alpha,\delta)$-Sasaki manifold splits into the horizontal and vertical subspaces $\mathcal{H}$ and $\mathcal{V}$. In the non-parallel case $\mathcal{V}$ is irreducible, in the parallel case it decomposes into $3$ trivial $1$-dimensional representations. In either case the curvature properties stated in \autoref{curv-ids} allow us to prove:

\begin{lem}\label{Vvertical}
The vertical distribution $\mathcal{V}$ of a $3$-$(\alpha,\delta)$-Sasaki manifold is vertical with respect to the above notation.
\end{lem}

\begin{proof}
By the Ambrose-Singer Theorem the holonomy algebra $\mathfrak{hol}$ of the holonomy group $\mathrm{Hol}(p)$ at a point $p$ is given by
\[
\mathfrak{hol}=\{P^{-1}_{\gamma}\circ R(P_\gamma X,P_\gamma Y)\circ P_{\gamma}\ |\ \gamma\ \text{some path from $p$ to $q$,}\ X,Y\in T_pM\}\subset\mathfrak{so}(T_pM)
\]
where $P_{\gamma}$ denotes parallel transport along $\gamma$ and $R(X,Y)\in\so(T_qM)$ the curvature operator. The horizontal and vertical distribution are invariant under parallel transport with respect to the canonical connection. Thus, we may assume $\gamma$ to be trivial when investigating the holonomy action on these distributions. By \eqref{0Curv} we know that the holonomy is only non-trivial if $X,Y\in\mathcal{V}$ or $X, Y\in\mathcal{H}$. In the first case \eqref{1Curv} and \eqref{offdiag} show that every element of $\mathfrak{hol}$ acting non-trivially on $\mathcal{V}$ must also act non-trivially on $\mathcal{H}$. The action of an element $R(X,Y)$, $X,Y\in\mathcal{H}$, on $\mathcal{V}$ is again given by \eqref{offdiag}. Any such element of $\mathfrak{hol}$ acts non-trivially on $\mathcal{V}$ if $\beta\neq 0$ and $\Phi_i(X,Y)\neq 0$ for some $i=1,2,3$. In this case $R(X,Y)$ is also a non-trivial operator on $\mathcal{H}$ by \eqref{2Curv}.
\end{proof}
\begin{prop}\label{prop.canonicalsubm}
The decomposition $TM=\mathcal{H}\oplus\mathcal{V}$ of a $3$-$(\alpha,\delta)$-Sasaki manifold $M$ satisfies the conditions in \autoref{CMSThm}. In particular, there exists a locally defined Riemannian submersion $\pi\colon M\to N$ such that
\begin{equation}\label{nablanablag}
\nabla^{g_N}_XY=\pi_*(\nabla_{\overline{X}}\overline{Y}).
\end{equation}
\end{prop}
\begin{definition}
We will call $\pi\colon M\to N$ the \emph{canonical submersion} of a $3$-$(\alpha,\delta)$-Sasaki manifold.
\end{definition}
\begin{proof}[Proof (of \autoref{prop.canonicalsubm})]
By \eqref{canonical} and \autoref{Vvertical} the decomposition $TM=\mathcal{V}\oplus\mathcal{H}$ is of type \eqref{tangentdecomposition} with respect to the canonical connection $\nabla$. By \eqref{torsion01} the projection of the torsion onto $\mathcal{H}\otimes\Lambda^2\mathcal{V}$ vanishes, satisfying \eqref{projecttau}. Therefore the conditions of \autoref{CMSThm} are satisfied. Moreover, \eqref{torsion01} shows that the projection of $\tau$ onto $\Lambda^3\mathcal{H}$ vanishes so the connection $\nabla^\sigma$ in \eqref{pinabla} for the canonical submersion is the Levi-Civita connection $\nabla^{g_N}$ on $N$.
\end{proof}

We observe that the canonical submersion is, indeed, an almost contact metric $3$-submersion in the sense of \cite{Watson84}, although we never make explicity use of this property (our formulas are much more detailed than the general results obtained therein).
\subsection{The quaternionic K\"ahler structure on the base}
%
We give a preliminary lemma needed to prove that the base of the canonical submersion admits a qK structure. Recall that a basic vector field on $M$ is a horizontal vector field which is projectable, that is $\pi$-related to some vector field defined on $N$. If $X\in TN$, the horizontal lift of $X$ is the unique basic vector field $\overline{X}\in TM$ such that $\pi_* \overline{X}=X$.
\begin{lemma}\label{Vnabla}
For any vertical vector field $X\in\mathcal{V}$ and for any basic vector field $Y\in\mathcal{H}$ we have
\[
(\nabla_{X}Y)_{\mathcal{H}}=-2\alpha\sum_{i=1}^3\eta_i(X)\varphi_i Y.
\]
\end{lemma}

\begin{proof}
We first use the identity $g(\nabla^{g}_{X}Y,Z)=-\frac 12 g([Y,Z],X)$ for any vector fields $X\in\mathcal{V}, Y,Z\in\mathcal{H}$, with $Y$ and $Z$ projectable, of a Riemannian submersion \cite[Proposition $13$]{Petersen}. Note that the horizontal and vertical distributions of the Riemannian submersion agree with the same notion in the $3$-$(\alpha,\delta)$-Sasaki setting. Further, we make use of \autoref{VComm} to obtain
\[
g(\nabla^g_{\xi_i}Y,Z)=-\frac 12 g([Y,Z],\xi_i)=\alpha\Phi_i(Y,Z).
\]
Therefore
\begin{align*}
g(\nabla_{X}Y,Z)&=g(\nabla^g_{X}Y,Z)+\frac 12 T(X,Y,Z)=\sum_{i=1}^3\eta_i(X)(\alpha\Phi_i(Y,Z)+\alpha\Phi_i(Y,Z))\\
&=-2\alpha\sum_{i=1}^3\eta_i(X)g(\varphi_i Y,Z).\qedhere
\end{align*}
\end{proof}

\begin{theorem}\label{qkbase}
The base $N$ of the canonical submersion $\pi\colon M\to N$ of any $3$-$(\alpha,\delta)$-Sasaki manifold $M$ carries a quaternionic K\"ahler structure given by
\[
\check{\varphi_i}=\pi_*\circ\varphi_i\circ s_*,\quad i=1,2,3,
\]
where $s\colon U\to M$ is any local smooth section of $\pi$. The covariant derivatives of the almost complex structures $\check{\varphi}_i$ are given by
\[
\nabla^{g_N}_X\check{\varphi}_i=2\delta(\check{\eta}_k(X)\check{\varphi}_j-\check{\eta}_j(X)\check{\varphi}_k),
\]
where $\check{\eta}_i(X)=\eta_i(s_*X)$ for $i=1,2,3$.
\end{theorem}

\begin{proof}
Let $s$ be a local section of the canonical submersion  $\pi\colon M\to N$, hence $\pi_*\circ s_*=\mathrm{id}$ and $\operatorname{Im}(s_*\circ\pi_*-\mathrm{id})\subset\mathcal{V}$ on the image $s(N)\subset M$. Define
\[
\check{\varphi}_i=\pi_*\circ \varphi_i\circ s_*
\]
for $i=1,2,3$. The horizontal and vertical distributions, $\mathcal{H}$ and $\mathcal{V}$, are invariant under $\varphi_i$. Thus, $\pi_*\circ\varphi_i=\check{\varphi_i}\circ\pi_*$ on $s(N)$. This yields
\begin{align*}
\check{\varphi}_i\check{\varphi}_j&=\check{\varphi}_i\circ (\pi_*\circ \varphi_j\circ s_*)
=\pi_*\circ(\varphi_i \varphi_j) \circ s_*.
\end{align*}
Now use that $(\varphi_i|_\mathcal{H})^2=-\mathrm{id}|_\mathcal{H}$ and $(\varphi_i|_\mathcal{H})(\varphi_j|_\mathcal{H})=\pm\varphi_k|_\mathcal{H}$ with sign $\pm$ depending on whether $(ijk)$ is an even or odd permutation of $(123)$. This shows $\check{\varphi}_i^2=-\mathrm{id}$ and $\check{\varphi}_i\check{\varphi}_j=\pm\check{\varphi}_k$.

Finally, by means of \eqref{nablanablag} and \eqref{canonical}, we show that the quaternionic structure is parallel. First
\begin{align*}
(\nabla^{g_N}_X\check{\varphi}_i)Y
&=(\nabla^{g_N}_X(\check{\varphi}_iY))-(\check{\varphi}_i(\nabla^{g_N}_XY))
=\pi_*\nabla_{\overline{X}}\overline{(\check{\varphi}_iY)}-\check{\varphi}_i\left(\pi_*(\nabla_{\overline{X}}\overline{Y})\right)\\
&=\pi_*\nabla_{\overline{X}}\overline{\left(\pi_*\left(\varphi_i( s_*Y)\right)\right)}-\pi_*\left(\varphi_i\left( s_*\left( \pi_*\left(\nabla_{\overline{X}}\overline{Y}\right)\right)\right)\right).
\end{align*}
By the properties of any Riemannian submersion we have that
$\overline{\left(\pi_*\left(\varphi_i( s_*Y)\right)\right)}=(\varphi_i( s_*Y))_{\H}$ wherever the right side is defined,
that is on the image $s(N)\subset M$. Thus, we take the covariant derivatives in the direction of $s_*X$ resulting
in a vertical correction term $\hat{X}=\overline{X}-s_*X\in\mathcal{V}$. Recall that $\nabla$ and $\varphi_i$ preserve
the horizontal and the vertical distribution. Using \autoref{Vnabla}, we obtain
\begin{align*}
\nabla_{\overline{X}}\overline{\left(\pi_*\left(\varphi_i( s_*Y)\right)\right)}
&=\nabla_{s_*X}(\varphi_i( s_*Y))_{\H}+\nabla_{\hat{X}}(\overline{\left(\pi_*\left(\varphi_i( s_*Y)\right)\right)})\\
&=\left(\nabla_{s_*X}\left(\varphi_i( s_*Y)\right)\right)_{\mathcal{H}}-2\alpha\sum_{l=1}^3\eta_l(\hat{X})\varphi_l \left(\varphi_i( s_*Y)\right)_{\H}.
\end{align*}
For the second summand, the horizontal projection is given by
\begin{align*}
\left(\varphi_i\left( s_*\left( \pi_*\left(\nabla_{\overline{X}}\overline{Y}\right)\right)\right)\right)_{\mathcal{H}}
&=\varphi_i\left( s_*\left( \pi_*\left(\nabla_{\overline{X}}\overline{Y}\right)\right)\right)_{\mathcal{H}}\\
&=\varphi_i\left(\nabla_{s_*X}(s_*Y)\right)_{\mathcal{H}}+\varphi_i\nabla_{\hat{X}}(\overline{Y})\\
&=\left(\varphi_i\left(\nabla_{s_*X}(s_*Y)\right)\right)_{\mathcal{H}}-2\alpha\sum_{l=1}^3\eta_l(\hat{X})\varphi_i(\varphi_l (s_*Y))_{\H}.
\end{align*}
Recombining both identities we obtain
\begin{align*}
(\nabla^{g_N}_X\check{\varphi}_i)Y
&=\pi_*\big(\nabla_{s_*X}\left(\varphi_i( s_*Y)\right)-\varphi_i\left(\nabla_{s_*X}(s_*Y)\right)+2\alpha\sum_{l=1}^3\eta_l(\hat{X})(\varphi_i\varphi_l (s_*Y)_{\H}-\varphi_l\varphi_i(s_*Y)_{\H})\big)\\
&=\pi_*\big((\nabla_{s_*X}\varphi_i)s_*Y-2\alpha\sum_{l=1}^3\eta_l(s_*X)(\varphi_i\varphi_l (s_*Y)_{\H}-\varphi_l\varphi_i(s_*Y)_{\H})\big)\\
&=(\beta+4\alpha)\big((\eta_k(s_*X)\circ s)\pi_*(\varphi_j(s_*Y))-(\eta_j(s_*X)\circ s)\pi_*(\varphi_k(s_*Y))\big)\\
&=2\delta(\check\eta_k(X)\check{\varphi}_j-\check\eta_j(X)\check{\varphi}_k)Y.
\end{align*}
Here we used the defining identity \eqref{canonical} of the canonical connection for any even permutation
$(ijk)$ of $(123)$. Therefore, the quaternionic structure is parallel and $N$ is quaternionic K\"ahler.
\end{proof}
\begin{rem}
A priori the quaternionic structure may depend on the chosen section $s$. Indeed, the individual almost complex
structures $\check{\varphi}_i$ vary with $s$. However, following the work of P. Piccinni and I. Vaisman \cite{PicVai},
one can see that the quaternionic structure is preserved under the Bott connection $\mathring{D}\colon\V \times\H\to \H$
defined by $\mathring{D}_VX=[V,X]_{\mathcal{H}}$, since
\[
(\mathring{D}_{\xi_i}\varphi_j)X=[\xi_i,\varphi_j X]_{\mathcal{H}}-\varphi_j[\xi_i,X]_\mathcal{H}=((\mathcal{L}_{\xi_i}\varphi_j)X)_\mathcal{H}=2\delta\epsilon_{ijk}\varphi_kX.
\]
This implies that the quaternionic structure is projectable and, thus, independent of choices.
\end{rem}
\begin{cor}
A $3$-$(\alpha,\delta)$-Sasaki manifold fibers locally over a hyperk\"ahler manifold if it is degenerate.
\end{cor}
\begin{rem}
Apart from the degenerate case the induced quaternionic K\"ahler structure is hyperk\"ahler if and only if $s_*X\in \mathcal{H}$ for all $X\in TN$. Such a section exists if and only if the horizontal distribution is tangent to $s(N)$ and, thus, integrable. This is in contrast to \autoref{VComm} for any $3$-$(\alpha,\delta)$-Sasaki manifold.
\end{rem}
We can now relate the curvature of $N$ with that of $M$.
\begin{theorem}\label{scal}
Let $\pi\colon M\to N$ be the canonical submersion of a $3$-$(\alpha,\delta)$-Sasaki manifold. Then
\[
\mathrm{scal}_{g_N}=16n(n+2)\alpha\delta.
\]
\end{theorem}
\begin{proof}
By \autoref{ATensor} and \eqref{torsion01} for $X,Y\in\mathcal{H}$
\[
\mathcal{A}_XY=-\frac 12 T(X,Y)_\mathcal{V}=-\alpha\sum_{i=1}^3 \Phi_i(X,Y)\xi_i.
\]
Let $e_1,\dots, e_{4n}$ be a local adapted frame for $\mathcal{H}$, i.e. an orthonormal frame such that
$\varphi_1e_{4p+1}=e_{4p+2}$, $\varphi_2e_{4p+1}=e_{4p+3}$ and $\varphi_3e_{4p+1}=e_{4p+4}$, $0\leq p\leq n-1$. Then
\[
\sum_{i,j=1}^{4n}g(A_{e_i}e_j,A_{e_i}e_j)=\alpha^2\sum_{i,j=1}^{4n}\sum_{k=1}^3\Phi^2_k(e_i,e_j)=\alpha^2\cdot 3\cdot 4n=12n\alpha^2.
\]
From the O'Neill identities one obtains for submersions with $\mathcal{T}=0$ the Ricci curvature identity
\cite[Proposition 9.36]{BesseEinstein}
\[
\mathrm{Ric}_g(X,Y)=\pi^*\mathrm{Ric}_{g_N}(X,Y)-2\sum_{j=1}^{4n}g(A_Xe_j,A_Ye_j).
\]
The Ricci curvature of $M$ is
$\mathrm{Ric}_g=2\alpha(2\delta(n+2)-3\alpha)g+2(\alpha-\delta)((2n+3)\alpha-\delta)g|_{\mathcal{V}}$ by
\cite[Proposition 2.3.3]{AgrDil}. Combining both identities we have
\begin{align*}
\mathrm{scal}_{g_N}&=\sum_{i=1}^{4n}\mathrm{Ric}_{g_N}(\pi_*e_i,\pi_*e_i)=\sum_{i=1}^{4n}\mathrm{Ric}_g(e_i,e_i)
+2\sum_{i,j=1}^{4n}g(A_{e_i}e_j,A_{e_i}e_j)\\
&=4n\cdot 2\alpha(2\delta(n+2)-3\alpha)+24n\alpha^2=16n(n+2)\alpha\delta.\qedhere
\end{align*}
\end{proof}
\begin{rem}
In particular, we recover the scalar curvature result $\mathrm{scal}_{g_N}=16n(n+2)$ known in the $3$-Sasaki case \cite[Theorem 13.3.13]{Boyer&Galicki}.
\end{rem}
%
\section[Construction of homogeneous $3$-$(\alpha,\delta)$-Sasaki manifolds]{Construction of non-degenerate homogeneous $3$-$(\alpha,\delta)$-Sa\-saki manifolds}
%
For homogeneous $3$-$(\alpha,\delta)$-Sasaki manifolds the canonical submersion is invariant. Hence, the base
$N$ is a homogeneous qK space. In the non-degenerate case \autoref{scal} shows that $N$ is a homogeneous quaternionic
K\"ahler space of non-vanishing scalar curvature. There are two families of such spaces known: Compact qK symmetric
spaces, named Wolf spaces, their non-compact duals and Alekseevsky spaces. The latter are homogeneous qK spaces
admitting a solvable transitive group action. D.~Alekseevsky conjectured that all homogeneous qK spaces with
negative scalar curvature are Alekseevsky spaces \cite{Alekseevsky}. In particular, the class of non-compact qK
symmetric spaces is included in the class of Alekseevsky.
We will give independent constructions of homogeneous $3$-$(\alpha,\delta)$-Sasaki manifolds over symmetric base
spaces and such fibering over Alekseevsky spaces.
\subsection{Homogeneous $3$-$(\alpha,\delta)$-Sasaki manifolds over symmetric
quaternionic K\"ahler spaces}
%
%
Let $G/G_0$ be a real symmetric space, i.e. $\mathfrak g=\mathfrak g_0 \oplus\mathfrak g_1$ with
$[\mathfrak g_i,\mathfrak g_j]\subset \mathfrak g_{i+j}$ on the level of Lie algebras. Suppose there exists a
connected subgroup $H\subset G_0$ such that $\mathfrak g_0$ splits into a direct sum of Lie algebras
$\mathfrak g_0=\mathfrak{h}\oplus\mathfrak{sp}(1)$. Finally, assume that
$\mathfrak g_1^{\mathbb{C}}=\mathbb C^2\otimes_{\mathbb{C}} W$, for some $\mathfrak h^{\mathbb{C}}$-module $W$ of
$\dim_\mathbb{C}W=2n$, and the adjoint action of $\mathfrak g_0^\mathbb{C}$ is given by
\[
\mathfrak h^{\mathbb{C}} \oplus \mathfrak{sp}(1)^{\mathbb{C}} \ni (A,B)\cdot ((z_1,z_2)\otimes w)
=B(z_1,z_2)\otimes w + (z_1,z_2)\otimes Aw,
\]
where $\mathfrak{sp}(1)^{\mathbb{C}}=\mathfrak{su}(2)^{\mathbb{C}}=\mathfrak{sl}(2,\mathbb{C})$ acts by multiplication
on $\mathbb C^2$. We will call $(G,G_0,H)$ \emph{generalized 3-Sasaki data}.
\begin{rem}
\begin{enumerate}[a)]
\item For compact $G$ this is called $3$-Sasaki data in \cite[Definition 12, p. 12]{Draperetall}.
\item Consider the homogeneous space $M=G/H$. The assumptions above imply that
$\mathfrak{g}=\mathfrak{h}\oplus\mathfrak{m}$ with $\mathfrak{m}=\mathfrak{sp}(1)\oplus\mathfrak{g}_1$ is a
reductive decomposition. We rename the spaces $\V=\mathfrak{sp}(1)$ and $\H=\mathfrak{g}_1$
to express their role as vertical and horizontal subspaces of a $3$-$(\alpha,\delta)$-Sasaki manifold via
$T_pM\cong \mathfrak{m}$. For clarity we restate the bracket relations between all these spaces. We have
$\mathfrak{g}=\mathfrak{h}\oplus\V\oplus\H$, where $\mathfrak{h}$ and $\V$ are
commuting subalgebras. Thus they form the joint subalgebra
$\mathfrak{h}\oplus\V=\mathfrak{g}_0\subset\mathfrak{g}$. The full set of commutator relations is
\begin{gather*}
[\mathfrak{h},\mathfrak{h}]\subset\mathfrak{h}\qquad
[\V,\V]\subset\V\qquad
[\mathfrak{h},\V]=0\\
[\mathfrak{h},\H]\subset\H\qquad
[\V,\H]\subset\H\qquad
[\H,\H]\subset\V\oplus\mathfrak{h}.
\end{gather*}
In particular, both $\V$ and $\H$ are $\mathfrak{h}$-invariant.
\item Since $G/G_0$ is a symmetric space
there exists a dual symmetric space $G^*/G_0$ for every generalized $3$-Sasaki data $(G, G_0, H)$. The Lie algebras can then be identified as
\begin{equation}\label{subspaces}
\mathfrak{g}^*=\mathfrak{h}\oplus\V\oplus i\H\subset\mathfrak{g}^\mathbb{C}.
\end{equation}
It is then clear that $(G^*,G_0,H)$ is generalized $3$-Sasaki data as well. This yields pairs of compact and non-compact generalized $3$-Sasaki data. For clarity
we will denote the compact top Lie group by $G$ and the non-compact one by $G^*$.
\item By \cite{Draperetall} any $3$-Sasaki data gives rise to a homogeneous $3$-Sasaki manifold. They were
completely determined in \cite{BGM94} by the fact that they are fiber-bundles over the quaternionic K\"ahler base
space $G/G_0$. The non-compact $G^*$ are thus given as the isometry group of the non-compact
quaternionic K\"ahler
symmetric spaces \cite[p. 409]{BesseEinstein}. Alltogether, we obtain \autoref{table.3-sasaki-data}.

\smallskip

\raa{1.3}
\begin{table}[h]
\[\begin{array}{cccccc}
\toprule
\boldsymbol{G} & \boldsymbol{G^*} & \boldsymbol{H} & \boldsymbol{G_0} & \boldsymbol{\dim} &\\
\midrule
\mathrm{Sp}(n+1) & \mathrm{Sp}(n,1) & \mathrm{Sp}(n) & \mathrm{Sp}(n)\mathrm{Sp}(1) & 4n+3 & n\geq0\\
\midrule
\mathrm{SU}(n+2) & \mathrm{SU}(n,2) & S(\mathrm{U}(n)\times \mathrm{U}(1)) & S(\mathrm{U}(n)\times \mathrm{U}(2)) & 4n+3 &n\geq 1\\
\midrule
\mathrm{SO}(n+4) & \mathrm{SO}(n,4) & \mathrm{SO}(n)\times \mathrm{Sp}(1) & \mathrm{SO}(n)\mathrm{SO}(4) & 4n+3 & n\geq 3\\
\midrule
\mathrm{G}_2 & \mathrm{G}_2^2 & \mathrm{Sp}(1) & \mathrm{SO}(4) & 11&\\
\midrule
\mathrm{F}_4 & \mathrm{F}_4^{-20} & \mathrm{Sp}(3) &\mathrm{Sp}(3)\mathrm{Sp}(1) &31&\\
\midrule
\mathrm{E}_6 & \mathrm{E}_6^2 & \mathrm{SU}(6) & \mathrm{SU}(6)\mathrm{Sp}(1) & 43 & \\
\midrule
\mathrm{E}_7 & \mathrm{E}_7^{-5} & \mathrm{Spin}(12) & \mathrm{Spin}(12)\mathrm{Sp}(1) & 67 & \\
\midrule
\mathrm{E}_8 & \mathrm{E}_8^{-24} & \mathrm{E}_7 &\mathrm{E}_7\mathrm{Sp}(1) &115 &\\
\bottomrule
\end{array}
\]
\caption{Complete table of generalized $3$-Sasaki data}\label{table.3-sasaki-data}
\end{table}

\end{enumerate}
\end{rem}
\begin{theorem}\label{construction}
Consider some generalized $3$-Sasaki data $(G,G_0,H)$ and $0\neq\alpha,\delta\in\R$. Additionally suppose $\alpha\delta>0$ if
$G$ is compact and $\alpha\delta<0$ if $G$ is non-compact.

Let $\kappa(X,Y)=\mathrm{tr}(\mathrm{ad}(X)\circ\mathrm{ad}(Y))$ denote the Killing form on $\mathfrak g$.
Then define the inner product $g$ on the tangent space $T_pM=T_p(G/H)\cong\mathfrak{m}$ by
\begin{align*}
g\vert_{\V}&=\frac{-\kappa}{4\delta^2(n+2)},\qquad
g\vert_{\H}=\frac{-\kappa}{8\alpha\delta (n+2)},\qquad
\V\perp\H.
\end{align*}
Let $\xi_i=\delta \sigma_i\in\V=\mathfrak{sp}(1)$, where the $\sigma_i$ are the elements of
$\mathfrak{sp}(1)=\mathfrak{su}(2)$ given by
\[
\sigma_1=
\begin{pmatrix}
i &0\\ 0&-i
\end{pmatrix},\quad
\sigma_2=
\begin{pmatrix}
0 &-1\\ 1&0
\end{pmatrix},\quad
\sigma_3=
\begin{pmatrix}
0&-i\\ -i &0
\end{pmatrix}.
\]
Define endomorphisms $\varphi_i\in\mathrm{End}_{\mathfrak{h}}(\mathfrak{m})$ for $i=1,2,3$ by
\begin{align*}
\varphi_i\vert_{\V}&=\frac 1 {2\delta} \mathrm{ad}\,\xi_i,\qquad
\varphi_i\vert_{\H}=\frac 1 \delta \mathrm{ad}\,\xi_i.
\end{align*}
Together with $\eta_i=g(\xi_i,\cdot)$ the collection $(G/H,\varphi_i,\xi_i,\eta_i,g)$ defines a
homogeneous $3$-$(\alpha,\delta)$-Sasaki structure.
\end{theorem}
Before we proceed with the proof, we collect some observations.
\begin{rem}
\begin{enumerate}[a)]\label{Rem_Constr}
\item In case $\alpha\delta>0$ the given $3$-$(\alpha,\delta)$-Sasaki structure is obtained via a $\H$-homothetic
deformation with parameters $a=\frac 1 {\alpha\delta}$, $b=\frac 1 {\delta ^2}-\frac 1 {\alpha\delta}$,
$c=\frac 1 \delta$ from the $3$-Sasaki structure given in \cite{Draperetall}.
\item Consider a homogeneous $3$-$(\alpha,\delta)$-Sasaki manifold $(G/H, \varphi_i, \xi_i, \eta_i, g)$ with
$\alpha\delta>0$ such that the iso\-tropy group $H$ is connected, i.e. $G/H\neq \mathbb{RP}^{4n+3}$. Then a $\H$-homothetic
deformation with $a=\alpha\delta$ and  $c=\delta$ induces a homogeneous $3$-Sasaki manifold with
 $G/H\neq \mathbb{RP}^{4n+3}$ and thus is given by the model in \cite{Draperetall}. By definition of $\mathcal H$-homothetic
deformations the above inverse deformation will restore the original objects. Thus, $(G/H, \varphi_i, \xi_i, \eta_i, g)$
is given by the construction in the theorem.
\item  Usually the real representation $\mathfrak{g}_1$ of $\mathfrak{h}$ will be irreducible and will only become
reducible when complexified, 
thus we cannot describe the action of $\V=\mathfrak{sp}(1)$ on $\H$ easily, but from the
complexified action we still find that the relations $\mathrm{ad}\,\xi_i^2=-\delta^2\mathrm{id}$ and
$\mathrm{ad}\,\xi_i\circ\mathrm{ad}\,\xi_j=\pm\delta\mathrm{ad}\,\xi_k$ when $(ijk)$ is an even, resp.\
odd permutation of $(123)$ hold on $\H$. 
\item The Riemannian metric on $\H$ is a fixed multiple of the Killing form on $\mathfrak{g}$ and thus
 the projection onto the symmetric orbit space
\[
G/H\rightarrow G/G_0
\]
is a Riemannian submersion. Indeed, this is the canonical submersion obtained in \autoref{qkbase}.
\item The real projective space $\mathbb{R}P^{4n+3}=\frac{\mathrm{Sp}(n+1)}{\mathrm{Sp}(n)\times\mathbb{Z}_2}$ and its non compact dual $\frac{\mathrm{Sp}(n,1)}{\mathrm{Sp}(n)\times\mathbb{Z}_2}$ also admit $3$-$(\alpha,\delta)$-Sasaki structures. They are obtained as the quotient of $S^{4n+3}=\frac{\mathrm{Sp}(n+1)}{\mathrm{Sp}(n)}$, resp.\ $\frac{\mathrm{Sp}(n,1)}{\mathrm{Sp}(n)}$ by the action of $\mathbb{Z}_2$ inside the fiber. Since the action is discrete these spaces cannot be discerned in the Lie algebra picture. Note that all relevant tensors are invariant under the $\mathbb{Z}_2$ action and thus local results obtained for $S^{4n+3}=\frac{\mathrm{Sp}(n+1)}{\mathrm{Sp}(n)}$, resp.\ $\frac{\mathrm{Sp}(n,1)}{\mathrm{Sp}(n)}$, remain true on $\mathbb{R}P^{4n+3}$ and its non compact dual.
\item Since the metric is a multiple of the Killing form and the Killing form is $\mathrm{ad}$-invariant $[X,\,\cdot\;]$ will be metric if it preserves $\mathcal{H}$ and $\mathcal{V}$. This is precisely the case if $X\in\mathcal{V}$. For $X\in\mathcal{H}$, we compute with $Y\in\mathcal{V}, Z\in\mathcal{H}$
\[
g([X,Y],Z)=\frac{-1}{8\alpha\delta (n+2)}\kappa([X,Y],Z)=\frac{1}{8\alpha\delta (n+2)}\kappa(Y,[X,Z])=-\frac{\delta}{2\alpha}g(Y,[X,Z]).
\]
Thus  $[X,\,\cdot\;]\in \mathfrak{so}(\mathfrak{m})$ if and only if $\delta=2\alpha$, i.e. we are in the parallel case. This is exactly the condition that our homogeneous space is naturally reductive. This can only occur if $\alpha\delta>0$, i.e. we are in the positive case.
\end{enumerate}
\end{rem}
\begin{proof}[Proof (of Theorem \ref{construction})]
If $G$ is compact, $\kappa<0$. If $G$ is of non-compact type, we have $\kappa|_\V<0$ while $
\kappa|_\H>0$ by \eqref{subspaces} . Thus, in both cases the given metric $g$ is indeed positive definite.

\autoref{Rem_Constr} shows $\mathrm{tr}(\mathrm{ad}^2\xi_i\vert_{\H})=\mathrm{tr}(-\delta^2\mathrm{id}\vert_{\H})=-4n\delta^2$. On $\V$ we have
\[
[\xi_i,[\xi_i,\xi_j]]=\pm2\delta[\xi_i,\xi_k]
=-4\delta^2\xi_j,
\]
whenever $(ijk)$ is an even, respectively odd, permutation of $(123)$. Thus, $\mathrm{tr}(\mathrm{ad}^2\xi_i\vert_{\V})=-8\delta^2$ and therefore
\[
g(\xi_i,\xi_i)= \frac{-\mathrm{tr}(\mathrm{ad}^2\xi_i)} {4\delta^2(n+2)}
=\frac{-\mathrm{tr}(\mathrm{ad}^2\xi_i\vert_{\H})-\mathrm{tr}(\mathrm{ad}^2\xi_i\vert_{\V})}{4\delta^2(n+2)}
=\frac{8\delta^2+4\delta^2n}{4\delta^2(n+2)}=1.
\]
On the contrary we have $\mathrm{tr}(\ad\xi_i\circ\ad\xi_j\vert_{\H})=\mathrm{tr}(\pm\ad\xi_k\vert_{\H})=0$
as its trace on the complexification vanishes. And similar $[\xi_i,[\xi_j,\xi_k]]=0$ if $(ijk)$ is any permutation
of $(123)$ or $[\xi_i,[\xi_j,\xi_k]]=4\delta^2\xi_j$ if $i=k\neq j$. In any case $\mathrm{tr}(\ad\xi_i\circ\ad\xi_j)=0$ and, hence, $g(\xi_i,\xi_j)=0$ if $i\neq j$.

Next we check that the endomorphisms $\varphi_i$ are metric almost complex structures on the complement to $\xi_i$. Note that they vanish on their corresponding $\xi_i$. Furthermore,
\begin{align*}
\varphi_i^2(\xi_j)&=\frac 1{4\delta^2}[\xi_i,[\xi_i,\xi_j]]=-\xi_j,\\
\varphi_i^2\vert_{\H}&=\frac 1{\delta^2}\ad\!^2\xi_i\vert_{\H}=\frac{-\delta^2}{\delta^2}\mathrm{id}=-\mathrm{id}.
\end{align*}
Since $\H$ and $\V$ are invariant under $\varphi_i$ we check orthogonality on each component individually. On $\H$ use the associativity of $\kappa$ to find
\[
\kappa(\varphi_iX,\varphi_iY)=-\kappa(X,\frac 1{\delta^2} \ad\!^2\xi_i Y)=\kappa(X,Y)
\]
and thus $g(\varphi_i X,\varphi_i Y)=\frac{-\kappa(\varphi_i X,\varphi_i Y)}{8\alpha\delta(n+2)}=\frac{-\kappa(X,Y)}{8\alpha\delta^2(n+2)}=g(X,Y)$. On $\V$ we have
\[
g(\varphi_i\xi_j,\varphi_i\xi_{j^\prime})=g(\frac 1{2\delta}\ad\xi_i(\xi_j),\frac 1{2\delta}\ad\xi_i(\xi_{j^\prime}))=g(\pm\xi_k,\pm\xi_{k^\prime})=g(\xi_j,\xi_{j^\prime})
\]
if $(ijk)$, $(ij^\prime k^\prime)$ are according permutations of $(123)$ and the left side vanishes whenever $j$ or $j^\prime$ equals $i$.

Next we check the compatibility conditions of the $3$ almost contact metric structures. Suppose $(ijk)$
is an even permutation of $(123)$ then $\varphi_i\xi_j=\xi_k$ and together with the invariance of $\H$ under $\varphi_i$ we conclude $\eta_i\circ\varphi_j=\eta_k$. Further, $\varphi_i\varphi_j\vert_{\H}=\frac 1{\delta^2} \ad\xi_i\circ\ad\xi_j\vert_{\H}=\frac 1\delta\ad\xi_k|_{\H}=\varphi_k|_{\H}$ and on $\V$ we have
\begin{align*}
\varphi_i\varphi_j\xi_i&=\frac1{4\delta^2}[\xi_i,[\xi_j,\xi_i]]=\xi_j=\varphi_k\xi_i=\varphi_k\xi_i+\eta_j(\xi_i)\xi_i,\\
\varphi_i\varphi_j\xi_j&=0=\xi_i-\xi_i=\varphi_k\xi_j-\eta_j(\xi_j)\xi_i,\\
\varphi_i\varphi_j\xi_k&=\frac 1{4\delta^2}[\xi_i,[\xi_j,\xi_k]]=\frac 1{2\delta}[\xi_i,\xi_i]=0=\varphi_k\xi_k+\eta_j(\xi_k)\xi_i.
\end{align*}
We have thus shown that the given structure is a homogeneous almost $3$-contact metric structure. It remains to show the $3$-$(\alpha,\delta)$-Sasaki condition $\dd\eta_i=2\alpha\Phi_i +2(\alpha-\delta)\eta_j\wedge\eta_k$, for any even permutation $(ijk)$ of $(123)$. We show this case by case. Note that the last summand vanishes whenever either entry is in $\H$. Let $X\in\H$. Then, since $\ad\xi_j X\in\H$,
\begin{align*}
\dd\eta_i(\xi_j,X)&=\xi_j(\eta_i(X))-X(\eta_i(\xi_j))-\eta_i(\ad\xi_j X)=-\eta_i(\ad\xi_j X)=0,\\
2\alpha\Phi_i(\xi_j,X)&=2\alpha g(\xi_j,\varphi_i X)=\frac {2\alpha}\delta g(\xi_j,\ad\xi_i X)=0.
\end{align*}
For $X,Y\in\H$ we use associativity of $\kappa$
\begin{align*}
\dd\eta_i(X,Y)&=X(\eta_i(Y))-Y(\eta_i(X))-\eta_i([X,Y])=-g(\xi_i,[X,Y])\\
&=\frac {1}{4\delta^2(n+2)}\kappa(\xi_i,[X,Y])=\frac {-1}{4\delta^2(n+2)}\kappa(\ad\xi_iY,X),\\
2\alpha\Phi_i(X,Y)&=2\alpha g(X,\varphi_i Y)=\frac{2\alpha}\delta g(X,\ad\xi_i Y)=\frac{-2\alpha}{8\alpha\delta^2(n+2)}\kappa(X,\ad\xi_i Y)\\
&=\frac{-1}{4\delta^2(n+2)}\kappa(X,\ad\xi_i Y).
\end{align*}
Finally, we have
\begin{equation}
\begin{aligned}\label{3adonV}
\dd\eta_i(\xi_j,\xi_k)&=\xi_j(\eta_i(\xi_k))-\xi_k(\eta_i(\xi_j))-\eta_i([\xi_j,\xi_k])=-\eta_i(2\delta\xi_i)=-2\delta,\\
2\alpha\Phi_i(\xi_j,\xi_k)&=2\alpha g(\xi_j,\varphi_i\xi_k)=- 2\alpha g(\xi_j,\xi_j)=- 2\alpha,\\
2(\alpha-\delta)\eta_{i+1}\wedge\eta_{i+2}(\xi_j,\xi_k)&=2(\alpha-\delta)=2\alpha-2\delta.\qedhere
\end{aligned}
\end{equation}
\end{proof}
%

\subsection{Negative homogeneous $3$-$(\alpha,\delta)$-Sasaki manifolds over Alekseevsky spaces}
%
In order to construct homogeneous $3$-$(\alpha,\delta)$-Sasaki manifolds we recall the setup in the unified construction of Alekseevsky spaces due to V.~Cort\'es \cite{CortesNewConstr}. Let $q\in \N$. Set $V=\R^{3,q}$ the real vector space with signature $(3,q)$. Let $\Cl^0(V)$ denote the even Clifford algebra over $V$. Depending on $q\mod 4$ there exist exactly one or two inequivalent
irreducible $\Cl^0(V)$-modules. Accordingly, let $l\in\N$, if $q\not\equiv 3\mod 4$, or $l^+,l^-\in\N$,
if $q\equiv 3\mod 4$. Then set
\[
\mathfrak{g}=\so(V)\oplus V\oplus \R D\oplus W,
\]
where $W$ is the sum of $l$ equivalent irreducible $\Cl^0(V)$-modules (or the sum of $l^+$, $l^-$
irreducible $\Cl^0(V)$-modules if there are two inequivalent ones) and $D$ a derivation with eigenvalue
decomposition $\so(V)\oplus V\oplus W$ and respective eigenvalues $(0,1,1/2)$.
The action of $\so(V)$ on $V$ is given by the standard representation and $\so(V)$
acts on $W$ via the isomorphism $\so(V)\cong\spin(V)\subset\Cl^0(V)$ $e\wedge e^\prime\mapsto -\frac{1}{2}ee^\prime$ if $e,e^\prime$ are orthogonal.
$V$ commutes with itself and $W$. Finally the commutators $[W,W]$ are given by some non-degenerate
$\so(V)$-equivariant map $\Pi\colon\Lambda^2W\to V$ where $\so(V)$ acts on $W$ as $\spin(V)$.

\begin{rem}
Note that $\Pi$ is unique up to rescaling along the irreducible summands of $W$ \cite[Theorem 5]{CortesNewConstr}.
This rescaling leads to an isomorphism of the Lie algebras $\mathfrak{g}(\Pi)$ and $\mathfrak{g}(\Pi^\prime)$
corresponding to two such maps $\Pi$ and $\Pi^\prime$. The isomorphism extends to an isomorphism of the
$3$-$(\alpha,\delta)$-Sasaki structures defined later on. Thus, we will ignore the ambiguity in $\Pi$ from here on.
\end{rem}

\begin{notation}
On $V=\R^{3,q}$ fix an ONB $\hat{e}_1,\hat{e}_2,\hat{e}_3,e_1,\dots,e_q$ with signature $(+,+,+,-,\dots, -)$.
Then with the identification $\so(V)\cong \Lambda^2 V$ we also obtain a standard basis of the space
$\so(V)$  given by $\{\hat{e}_i\wedge \hat{e}_j, \hat{e}_i\wedge e_k, e_k\wedge e_l\}_{{}^{i,j=1,2,3}_{k,l=1,\dots,3}}$.

Denote $\sigma_i=2\hat{e}_k\wedge\hat{e}_j$ for any even permutation $(ijk)$ of $(123)$. Using the identification
$\mathrm{End}(V)=V\otimes V^*$ this implies $[\sigma_i,\hat{e}_j]=2\hat{e}_k$ and $[\sigma_i,\sigma_j]=2\sigma_k$
where again $(ijk)$ is an even permutation of $(123)$.

We further set $\mathcal{V}=\so(3)\subset\so(3,q)$, $\mathcal{H}_0$ the subspace generated by the elements
$D$ and $\hat{e}_i+\sigma_i$ and $\mathcal{H}_1$ the subspace generated by $e_1,\dots, e_q\in V$ and $e_i\wedge \hat{e}_j\in \so(3,q)$.
\end{notation}

The $4$-dimensional spaces $\mathcal{H}_0$ and $\langle e_l, e_l\wedge\hat{e}_j\rangle\subset \mathcal{H}_1$
will form the quaternionic subspaces inside $\so(V)\oplus V\oplus \R D\subset\mathfrak{g}$.
Accordingly, we show that they have the only commutators with non-trivial $\mathcal{V}$-part.

\begin{lemma}\label{Commutators}
The only non-trivial projections on $\mathcal{V}$ of commutators are
\begin{gather*}
\pi_\mathcal{V}([\sigma_i,\sigma_j])=\pm2\sigma_k,\quad
\pi_\mathcal{V}([D,\hat{e}_i+\sigma_i])=-\sigma_i, \quad
\pi_\mathcal{V}([\hat{e}_i+\sigma_i,\hat{e}_j+\sigma_j])=\mp 2\sigma_k,\\
\pi_\mathcal{V}([e_l,\hat{e}_i\wedge e_l])=-\sigma_i,\quad
\pi_\mathcal{V}([\hat{e}_i\wedge e_l,\hat{e}_j\wedge e_l])=\pm\frac{1}{2}\sigma_k,\quad
\pi_\mathcal{V}([w_1,w_2])=\pi_\mathcal{V}(\Pi(w_1,w_2))
\end{gather*}
for all permutations $(ijk)$  of $(123)$ with $\pm$ indicating the sign of the permutation,
$l=1,\dots,q$ and $w_1,w_2\in W$.
\end{lemma}
\begin{proof}
The full list of commutators of basis vectors is
\begin{align*}
[\sigma_i,\sigma_j]&=\pm2\sigma_k, &
[\sigma_i,D]&=0, &
[\hat{e}_i+\sigma_i,\hat{e}_i\wedge e_l]&=e_l,&
[\sigma_i,\hat{e}_j+\sigma_j]&=\pm2(\hat{e}_k+\sigma_k),\\
[\sigma_i,e_l]&=0,&
[\sigma_i,\hat{e}_i\wedge e_l]&=0,&
[\hat{e}_i\wedge e_l, W]&=W, &
[\sigma_i,\hat{e}_j\wedge e_l]&=\pm2\hat{e}_k\wedge e_l,\\
[D,e_l]&=e_l,&
[D,\hat{e}_i\wedge e_l]&=0,&
[\hat{e}_i\wedge e_l,\hat{e}_j\wedge e_m]&=0,&
[D,\hat{e}_i+\sigma_i]&=\hat{e}_i=(\hat{e}_i+\sigma_i)-\sigma_i,\\
[D,W]&=W,&
[\sigma_i,\hat{e}_i+\sigma_i]&=0, &
[\hat{e}_i+\sigma_i,W]&=W,&
[\hat{e}_i+\sigma_i,\hat{e}_j\wedge e_l]&=\pm2\hat{e}_k\wedge e_l,\\
[\sigma_i,W]&= W,&
[\hat{e}_i+\sigma_i,e_l]&=0,&
[\hat{e}_i\wedge e_l,\hat{e}_i\wedge e_m]&=-e_l\wedge e_m,&
[\hat{e}_i\wedge e_l,\hat{e}_j\wedge e_l]&=-\hat{e}_i\wedge\hat{e}_j=\pm\frac{1}{2}\sigma_k\\
[e_l,e_m]&=0, &
[e_l, W]&=W, &
[e_l,\hat{e_i}\wedge e_m]&=0, &
[e_l,\hat{e_i}\wedge e_l]&=\hat{e_i}=(\hat{e_i}+\sigma_i)-\sigma_i
\end{align*}
and finally
\bdm
[\hat{e}_i+\sigma_i,\hat{e}_j+\sigma_j]\ = \
[\sigma_i,\hat{e}_j]-[\sigma_j,\hat{e}_i]+[\sigma_i,\sigma_j] \ = \
\pm4\hat{e}_k\pm2\sigma_k=\pm4(\hat{e}_k+\sigma_k)\mp2\sigma_k,
\edm
where $(ijk)$ is a permutation of $(123)$ with $\pm$ indicating the sign of the permutation and $l,m=1,\dots,q$ with $l\neq m$. For the commutator $[W,W]$ we have $[w_1,w_2]=\Pi(w_1,w_2)\in V\subset\mathcal{H}_0\oplus\mathcal{H}_1\oplus\mathcal{V}$.
\end{proof}

By \cite[Proposition 3]{CortesNewConstr} the adjoint action $\mathfrak{g}\curvearrowright \mathfrak{r}=\R D\oplus V\oplus W\subset \mathfrak{g}$ is faithful. Thus, $\mathfrak{g}$ is a subalgebra $\mathfrak{g}\subset \mathrm{der}(\mathfrak{r})$. Set $G$ the subgroup $G\subset \mathrm{Aut}(\mathfrak{r})$ with Lie Algebra $\mathfrak{g}$. Let $\mathfrak{h}=\so(q)\subset\so(V)\subset \mathfrak{g}$ and $H\subset G$ the corresponding connected subgroup. Then both $G$ and $H$ are closed subgroups of $\mathrm{Aut}(\mathfrak{r})$. This follows from \cite[Corollary 3]{CortesNewConstr} and the fact that $H$ is closed in $\mathrm{Spin}_0(V)\subset G$. In particular, $G/H$ is a homogeneous space. We now define the desired negative $3$-$(\alpha,\delta)$-Sasaki structure on $M=G/H$.

\begin{theorem}\label{ThmAlekseevsky}
Let $\alpha,\delta\in \R$ with $\alpha\delta<0$. Let $G,H$ with Lie algebras $\mathfrak{g},\mathfrak{h}$ as above. Then $\mathfrak{m}=\mathcal{V}\oplus\mathcal{H}_0\oplus\mathcal{H}_1\oplus W$ is a reductive complement to $\mathfrak{h}$ in $\mathfrak{g}$. Set
\[
\xi_1=\delta \sigma_1, \qquad\xi_2=\delta\sigma_2, \qquad\xi_3=\delta\sigma_3.
\]
Define the almost complex structures $\varphi_i\colon\mathfrak{m}\to\mathfrak{m}$ on $\mathcal{V}$, $\mathcal{H}_0$, $\mathcal{H}_1$ and $W$ individually.
For any permutation $(ijk)$ of $(123)$ with signature $\pm$ we set
\begin{gather*}
\tag{$\mathcal{V}$}\varphi_i(\sigma_j)=\pm\sigma_k,\qquad \varphi_i(\sigma_i)=0,\\
\tag{$\mathcal{H}_0$}\varphi_i(2D)=\hat{e}_i+\sigma_i,\qquad \varphi_i(\hat{e}_i+\sigma_i)=-2D,\qquad \varphi_i(\hat{e}_j+\sigma_j)=\pm(\hat{e}_k+\sigma_k),\\
\tag{$\mathcal{H}_1$}\varphi_i(e_l)=2\hat{e}_i\wedge e_l,\qquad \varphi_i(2\hat{e}_i\wedge e_l)=-e_l,\qquad \varphi_i(\hat{e}_j\wedge e_l)=\pm \hat{e}_k\wedge e_l,\\
\tag{$W$}\varphi_i|_W=\rho(\sigma_i),
\end{gather*}
where $\rho$ is the Clifford-multiplication on $W$.

Define a scalar product $g_{[e]}$ by declaring the following vectors to be an orthonormal basis of $\mathcal{V}\oplus \mathcal{H}_0\oplus \mathcal{H}_1$:
\begin{gather*}
\delta\sigma_i,
\sqrt{-4\alpha\delta}D,
\sqrt{-\alpha\delta}(\hat{e}_i+\sigma_i),
\sqrt{-4\alpha\delta}\,\hat{e}_i\wedge e_l,
\sqrt{-\alpha\delta}\,e_l.
\end{gather*}
On $W$ we set the scalar product
\[
g_{[e]}|_{W\times W}(s,t)=(-2\alpha\delta)^{-1}b(s,t)\coloneqq(-2\alpha\delta)^{-1}\langle \hat{e}_i,\Pi(\rho(\hat{e}_j\hat{e}_k)s,t)\rangle,
\]
where $\langle,\rangle$ is the scalar product on $V$ and $(ijk)$ is any even permutation of $(123)$. We set $W$ orthogonal to $\mathcal{V}\oplus\mathcal{H}_0\oplus\mathcal{H}_1$. Set $\eta_i=g(\xi_i,\,\cdot\,)$ the dual to $\xi_i$.

Then $(G/H,g,\xi_i,\eta_i,\varphi_i)$ defines a homogeneous $3$-$(\alpha,\delta)$-Sasaki manifold.
\end{theorem}

\begin{proof}
We first note that the defined scalar product is positive definite and $\mathrm{Spin}(q)$-invariant. This is clear on $\mathcal{V}\oplus\mathcal{H}_0\oplus\mathcal{H}_1$ and it is shown for $b$ in \cite[Theorem 1 and Proposition 9]{CortesNewConstr}. Thus, the scalar product extends to an invariant Riemannian metric on $G/H$. The invariance under $H$ of the $\xi_i$ is obvious. For an invariant $3$-a.c.m. structure, it remains to check that the $\varphi_i$ are invariant as well. $\mathrm{Spin}(q)$ acts trivial on $\mathcal{V}\oplus\mathcal{H}_0$ and on $\mathcal{H}_1$ by its adjoint action on $e_l\in \R^q\subset V$. On $W$ it acts by Clifford multiplication with vectors in $\R^q$ twice, thus commuting with the Clifford multiplication defining the almost complex structures on $W$.

The endomorphisms $\varphi_i$ are compatible with the metric by definition on $\mathcal{V}\oplus\mathcal{H}_0\oplus\mathcal{H}_1$ and by $\mathrm{Spin}(q)\cdot \mathrm{Spin}(3)$-invariance of $b$ on $W$. Next we check the compatibility conditions of the $3$ almost contact structures. Again on $\mathcal{V}\oplus\mathcal{H}_0\oplus\mathcal{H}_1$ this is a direct consequence of the definition and on $W$ we have
\[
\rho(\sigma_i)\rho(\sigma_j)w=\hat{e}_k\cdot\hat{e}_j\cdot\hat{e}_i\cdot\hat{e}_k\cdot w=(-1)^2\hat{e}_j\cdot\hat{e}_k\cdot\hat{e}_k\cdot\hat{e}_i\cdot w=-\hat{e}_j\cdot\hat{e}_i\cdot w=\rho(\sigma_k)w.
\]

Finally we need to check the defining condition $\mathrm{d}\eta_i=2\alpha\Phi_i+2(\alpha-\delta)\eta_j\wedge\eta_k$. By bilinearity it suffices to check it for any pair of two basis vectors individually. On $\mathcal{V}\times\mathcal{V}$ this is exactly the same computation as in the $3$-$(\alpha,\delta)$-Sasaki structure over symmetric bases (compare \eqref{3adonV}). Apart from $\mathcal{V}\times\mathcal{V}$ the equation reduces to $\mathrm{d}\eta_i=2\alpha\Phi_i$. Note that the left hand side reduces to checking the commutators. From Lemma \ref{Commutators} and the definition of the $\varphi_i$ we see that both sides vanish for all mixed terms regarding the decomposition $\mathcal{V}\oplus\mathcal{H}_0\oplus\mathcal{H}_1\oplus W$ of the tangent space. Similarly on $\mathcal{H}_1$ if the index $l$ of $\hat{e}_i\wedge e_l$, respectively $e_l$, is not the same both sides vanish. On $\mathcal{H}_0\times\mathcal{H}_0$ we compute
\begin{align*}
\mathrm{d}\eta_i(D,\hat{e}_i+\sigma_i)&=-\eta_i([D,\hat{e}_i+\sigma_i])=-\eta_i(-\sigma_i)=\frac{1}{\delta}g(\delta\sigma_i,\delta\sigma_i)=\frac{1}{\delta},\\
2\alpha\Phi_i(D,\hat{e}_i+\sigma_i)&=2\alpha g(D,\varphi_i(\hat{e}_i+\sigma_i))=\frac{2\alpha}{-2\alpha\delta} g(\sqrt{-4\alpha\delta}\;D,-\sqrt{-\alpha\delta}\;2D)=\frac{1}{\delta}.
\end{align*}
In similar fashion for the remaining pairs in $\mathcal{H}_0\times\mathcal{H}_0$ and on $\mathcal{H}_1\times\mathcal{H}_1$ we have
\begin{align*}
\frac{2}{\delta}=-\eta_k(2\sigma_k)=\mathrm{d}\eta_k(\hat{e}_i+\sigma_i,\hat{e}_j+\sigma_j)&=2\alpha\Phi_k(\hat{e}_i+\sigma_i,\hat{e}_j+\sigma_j)=\frac{2}{\delta},\\
\frac{1}{\delta}=-\eta_i(-\sigma_i)=\mathrm{d}\eta_i(e_l,\hat{e}_i\wedge e_l)&=2\alpha\Phi_i(e_l,\hat{e}_i\wedge e_l)=\frac{1}{\delta},\\
\frac{1}{2\delta}=-\eta_k(-\frac 12\sigma_k)=\mathrm{d}\eta_k(\hat{e}_i\wedge e_l,\hat{e}_j\wedge e_l)&=2\alpha\Phi_k(\hat{e}_i\wedge e_l,\hat{e}_j\wedge e_l)=\frac{1}{2\delta}
\end{align*}
for any even permutation $(ijk)$ of $(123)$.
Finally, we look at $W\times W$. Let $w_1,w_2\in W$ and suppose $\Pi(w_1,w_2)=\sum_{r=1}^q a_re_r+\sum_{s=1}^3\hat{a}_s\hat{e}_s$. Then
\begin{align*}
\mathrm{d}\eta_i([w_1,w_2])&=-\eta_i(\Pi(w_1,w_2))=-\eta_i\left(\sum_{r=1}^q a_re_r+\sum_{s=1}^3\hat{a}_s\hat{e}_s\right)=-\eta_i\left(\sum_{s=1}^3\hat{a}_s((\hat{e}_s+\sigma_s)-\sigma_s)\right)=\frac{\hat{a}_i}{\delta}
\end{align*}
and
\begin{align*}
2\alpha\Phi_i(w_1,w_2)&=2\alpha g(w_1,\varphi_i w_2)=\frac{2\alpha}{-2\alpha\delta}\left\langle\hat{e}_i,\Pi\left(w_1,\hat{e}_j\hat{e}_k\hat{e}_j\hat{e}_k w_2\right)\right\rangle=\frac{(-1)^32\alpha}{-2\alpha\delta}\langle\hat{e_i},\Pi(w_1,w_2)\rangle\\
&=\frac{1}{\delta}\left\langle\hat{e}_i,\sum_{r=1}^q a_re_r+\sum_{s=1}^3\hat{a}_s\hat{e}_s\right\rangle=\frac{\hat{a}_i}{\delta}.
\end{align*}
This concludes the proof.
\end{proof}

\begin{remark*}
We try to motivate the definition. Recall that in \cite{CortesNewConstr} Cort\'es shows that $\so(V)\oplus V\oplus\R D$ is isomorphic to a subalgebra of $\so(4,q+1)=\Lambda^2(V\oplus \langle e^+,e^-\rangle)$, $e^+,e^-$ unit length vectors of corresponding signature, given by the inclusion
\[
\so(V)\mapsto\Lambda^2 V,\qquad V\mapsto V\wedge (e^+-e^-),\qquad D\mapsto e^+\wedge e^-.
\]
Now $\varphi_i$ is modeled  on $\so(V)\oplus V\oplus\R D$ after the adjoint action with $\hat{e}_j\wedge\hat{e}_k+\hat{e}_i\wedge e^+\in \so(3)_+\subset \so(3)_+\oplus\so(3)_-=\so(4)$ in the known $\SO(4,q+1)/\SO(q+1)\SO(3)$ setting. However, this does not exist as an inner derivative in $\mathfrak{g}$ unlike int the (semi-) simple case.

\end{remark*} 
%
\subsection{Examples}
%
We begin with an example of the construction over a symmetric Wolf space.

\begin{ex}
Our first example is the Aloff-Wallach space $W^{1,1}=\mathrm{SU(3)}/S^1=G/H$. In this case the
isotropy algebra $\mathfrak{h}$ inside $\mathfrak{g}=\mathfrak{su}(3)$ is the 1-dimensional space generated by
\[
h=\begin{bmatrix}
-i & 0&0\\
0&-i&0\\
0&0&2i
\end{bmatrix}.
\]
We locate the space $\mathfrak{sp}(1)=\mathfrak{su}(2)\subset\mathfrak{su}(3)$ as the upper left
2-by-2 block. One checks that this is a splitting of $\mathfrak{su}(3)$ as necessary. Then for
$\alpha,\delta>0$ the Reeb vector fields are given by
\begin{gather*}
\xi_1=\delta
\begin{bmatrix}
i&0&0\\
0&-i&0\\
0&0&0
\end{bmatrix},\qquad
\xi_2=\delta
\begin{bmatrix}
0&-1&0\\
1&0&0\\
0&0&0
\end{bmatrix},\qquad
\xi_3=\delta
\begin{bmatrix}
0&-i&0\\
-i&0&0\\
0&0&0
\end{bmatrix}.
\end{gather*}
On the horizontal subspace we choose a basis vector
\begin{align*}
\tilde{e}_1=\begin{bmatrix}
0&0&1\\
0&0&0\\
-1&0&0
\end{bmatrix}.
\end{align*}
Then we normalize it $g(\tilde{e}_1,\tilde{e}_1)=\frac{-6\mathrm{tr}(\tilde{e}_1\cdot
\tilde{e}_1)}{24\alpha\delta}=\frac 1 {2\alpha\delta}$, i.e. $e_1=\sqrt{2\alpha\delta}\cdot\tilde{e}_1$
and generate an adapted basis:
\begin{gather*}
e_2=\sqrt{2\alpha\delta}
\begin{bmatrix}
0&0&i\\
0&0&0\\
i&0&0
\end{bmatrix},\qquad
e_3=\sqrt{2\alpha\delta}
\begin{bmatrix}
0&0&0\\
0&0&1\\
0&-1&0
\end{bmatrix},\qquad
e_4=\sqrt{2\alpha\delta}
\begin{bmatrix}
0&0&0\\
0&0&-i\\
0&-i&0
\end{bmatrix}.
\end{gather*}
\end{ex}

\begin{ex}
Next consider the dual negative $3$-$(\alpha,\delta)$-Sasaki space $\mathrm{SU}(2,1)/S^1$.
We realize the Lie algebra
\[
\mathfrak{su}(2,1)=\mathfrak{g}^*=\mathfrak{h}\oplus\mathfrak{sp}(1)\oplus i\H\subset \mathfrak{su}(3)^\mathbb{C}
\]
as described in \eqref{subspaces}. Then as for the
Aloff-Wallach space we identify the $1$-dimensional isotropy $\mathfrak{h}$ generated by
\[
h=\begin{bmatrix}
-i & 0&0\\
0&-i&0\\
0&0&2i
\end{bmatrix}.
\]
Analogously the Reeb vector fields are given by
\begin{gather*}
\xi_1=\delta
\begin{bmatrix}
i&0&0\\
0&-i&0\\
0&0&0
\end{bmatrix},\qquad
\xi_2=\delta
\begin{bmatrix}
0&-1&0\\
1&0&0\\
0&0&0
\end{bmatrix},\qquad
\xi_3=\delta
\begin{bmatrix}
0&-i&0\\
-i&0&0\\
0&0&0
\end{bmatrix}.
\end{gather*}
On the horizontal subspace we choose
\[
\tilde{e}^*_1=i\tilde{e}_1=\begin{bmatrix}
0&0&i\\ 0&0&0\\ -i&0&0
\end{bmatrix}\subset i\H.
\]
We have
\[
g(\tilde{e}^*_1,\tilde{e}^*_1)=\frac{-i^2}{24\alpha\delta}\kappa(\tilde{e}_1,\tilde{e}_1)=-\frac{1}{2\alpha\delta}.
\]
Thus we find an adapted base of $\mathrm{SU}(2,1)/S^1$ by $e_1^*=i\sqrt{-2\alpha\delta}\tilde{e}_1$ and
\begin{gather*}
e^*_2=\sqrt{-2\alpha\delta}
\begin{bmatrix}
0&0&-1\\
0&0&0\\
-1&0&0
\end{bmatrix},\;
e^*_3=\sqrt{-2\alpha\delta}
\begin{bmatrix}
0&0&0\\
0&0&i\\
0&-i&0
\end{bmatrix},\;
e^*_4=\sqrt{-2\alpha\delta}
\begin{bmatrix}
0&0&0\\
0&0&1\\
0&1&0
\end{bmatrix}.
\end{gather*}
\end{ex}
We now discuss the lowest dimensional example $\hat{\mathcal{T}}(1)$ of a negative homogeneous $3$-$(\alpha,\delta)$-Sasaki manifold fibering over Alekseevsky space $\mathcal{T}(1)$ not obtained by the construction over symmetric spaces.
\raa{1.3}
\begin{table}[h!]
\[
\begin{array}{cclll}
\toprule
\text{Dimension} &  \text{parameters} & \mathfrak{g} &  \mathfrak{h} & \text{alternative description}\\
\midrule
7 & q=0,\;l=0 &  \so(3)\oplus\R^{3}\oplus\R D & 0 & \mathrm{Sp}(1,1)/\mathrm{Sp}(1)\\
\midrule
\multirow{2}{*}{$11$} & q=0,\;l=1 & \so(3)\oplus\R^{3}\oplus\R D\oplus W_0&0&\mathrm{Sp}(2,1)/\mathrm{Sp}(2)\\
&q=1,\;l=0 & \so(3,1)\oplus\R^{3,1}\oplus\R D&0&\mathrm{SU}(2,2)/\mathrm{S}(\mathrm{U}(2)\times\mathrm{U}(1))\\
\midrule
\multirow{3}{*}{$15$} & q=0,\;l=2 & \so(3)\oplus\R^{3}\oplus\R D\oplus 2W_0& 0&\mathrm{Sp}(3,1)/\mathrm{Sp}(3)\\
&q=1,\;l=1 & \so(3,1)\oplus\R^{3,1}\oplus\R D\oplus W_1&0&\mathrm{SU}(3,2)/\mathrm{S}(\mathrm{U}(3)\times\mathrm{U}(1))\\
&q=2,\;l=0 & \so(3,2)\oplus\R^{3,2}\oplus\R D&\so(2)&\mathrm{SO}_0(3,4)/\mathrm{SO}(3)\times\mathrm{SO}(3)\\
\midrule
\multirow{4}{*}{$19$} & q=0,\;l=3 & \so(3)\oplus\R^{3}\oplus\R D\oplus 3W_0& 0 &\mathrm{Sp}(4,1)/\mathrm{Sp}(4)\\
& q=1,\;l=2 & \so(3,1)\oplus\R^{3,1}\oplus\R D\oplus 2 W_1& 0&\mathrm{SU}(4,2)/\mathrm{S}(\mathrm{U}(4)\times\mathrm{U}(1))\\
&q=2,\;l=1 & \so(3,2)\oplus\R^{3,2}\oplus\R D\oplus W_2&\so(2)&\hat{\mathcal{T}}(1)\ \text{non symmetric base}\\
&q=3,\;l=0 & \so(3,3)\oplus\R^{3,3}\oplus\R D & \so(3) & \mathrm{SO}_0(4,4)/\mathrm{SO}(4)\times\mathrm{SO}(3)\\
\bottomrule
\end{array}
\]
\caption{$3$-$(\alpha,\delta)$-Sasaki manifolds over Alekseevsky spaces of
$\dim\leq 19$, see \autoref{Alekseevskydimleq19}.}\label{table.Alekseevsky}
\end{table}
\begin{rem}\label{Alekseevskydimleq19}
The first new example arising from the construction over Alekseevsky spaces appears only in
dimension $19$. Table \ref{table.Alekseevsky} lists all homogeneous $3$-$(\alpha,\delta)$-Sasaki manifolds
obtained by \autoref{ThmAlekseevsky} up to dimension $19$ and, if existing, the isomorphic ones
appearing in \autoref{table.3-sasaki-data}, i.e.\ obtained by \autoref{construction} over non-compact symmetric spaces.
The list gets more intricate with higher dimension, in particular, there appear two inequivalent even Clifford modules for $q=3$ beginning in $\dim 27$ and for $q\geq 4$ we have $\dim W_q>4$.
Further, observe that the symmetric base cases $\mathrm{SU}(2,1)/\mathrm{U}(1)$, $\mathrm{G}_2^{(2)}/\mathrm{SO}(3)$ are not obtained by this construction.
\end{rem}
We now give more concrete descriptions of the $\Cl^0(3,q)$-modules $W_q$ for $q=0,1,2$. Note that there are
choices to be made though these lead to isomorphisms of the modules since all these modules are unique.
Let $\R^{3,q}=\langle e_{\hat{1}},e_{\hat{2}},e_{\hat{3}},e_1,\dots, e_q\rangle$, where $e_{\hat{i}}$ have signature
$+1$ while $e_i$ have signature $-1$. Then we have $\Cl^0(3,0)=\mathbb{H}$,
$\Cl^0(3,1)=\mathcal{M}_2(\mathbb{C})$, $\Cl^0(3,2)=\mathcal{M}_4(\mathbb{R})$ realized as follows.
\autoref{table.q=0and1} lists the cases $q=0$ and $q=1$, while \autoref{table.q=2} is devoted to the case
$q=2$.

\bigskip

\raa{1.3}
\begin{table}[h]
\[\begin{array}{@{}lllll@{}}
\toprule
q=0 & \deg 0: &  \begin{bsmallmatrix}1\end{bsmallmatrix} &  \phantom x &+1\\
\cmidrule(l){2-5}
    & \deg 2: & e_{\hat{1}\hat{2}}=\begin{bsmallmatrix}i\end{bsmallmatrix}, \quad
e_{\hat{2}\hat{3}}=\begin{bsmallmatrix}j\end{bsmallmatrix}, \quad
e_{\hat{3}\hat{1}}=\begin{bsmallmatrix}k\end{bsmallmatrix} & & -1 \\
\midrule[\heavyrulewidth]
q=1 & \deg 0: & \begin{bsmallmatrix}1&\\&1\end{bsmallmatrix} &  & +1\\
\cmidrule(l){2-5}
    & \deg 2: & e_{\hat{1}\hat{2}}=\begin{bsmallmatrix}i&\\&-i\end{bsmallmatrix}, \quad
e_{\hat{2}\hat{3}}=\begin{bsmallmatrix}&i\\i&\end{bsmallmatrix}, \quad
e_{\hat{3}\hat{1}}=\begin{bsmallmatrix}&-1\\1&\end{bsmallmatrix}& & -1 \\
    &        &  e_{\hat{3}1}=\begin{bsmallmatrix}1&\\&-1\end{bsmallmatrix},\quad
e_{\hat{1}1}=\begin{bsmallmatrix}&1\\1&\end{bsmallmatrix}, \quad
e_{\hat{2}1}=\begin{bsmallmatrix}&i\\-i&\end{bsmallmatrix}
& & +1\\
\cmidrule(l){2-5}
    & \deg 4: & e_{\hat{1}\hat{2}\hat{3}1}=\begin{bsmallmatrix}i&\\&i\end{bsmallmatrix} & &  -1\\
\bottomrule
\end{array}
\]
\caption{Choice of $\Cl^0(3,q)$-representations  for $q=0$ and $q=1$ }\label{table.q=0and1}
\end{table}

\bigskip

The notation is as follow: We denote elements $e_{ij}=e_ie_j\in\Cl^0(V)$ and analogous for the action of elements in $\Cl^0(V)$ of higher degree. The last line denotes the square of elements in the respective row, which are invariant of choices unlike the matrices itself.

\bigskip
\begin{table}[h]
\[\begin{array}{@{}lll@{}}
\toprule
\deg 0: & \begin{bsmallmatrix}1&&&\\&1&&\\&&1&\\&&&1\end{bsmallmatrix} & +1\\ \addlinespace[1mm]
\midrule
\deg 2: & e_{\hat{1}\hat{2}}=\begin{bsmallmatrix}&-1&&\\1&&&\\&&&1\\&&-1&\end{bsmallmatrix},\
e_{\hat{2}\hat{3}}=\begin{bsmallmatrix}&&&-1\\&&1&\\&-1&&\\1&&&\end{bsmallmatrix}, \
e_{\hat{3}\hat{1}}=\begin{bsmallmatrix}&&-1&\\&&&-1\\1&&&\\&1&&\end{bsmallmatrix}, \
e_{12}=\begin{bsmallmatrix}&&1&\\&&&-1\\-1&&&\\&1&&\end{bsmallmatrix} & -1\\ \addlinespace[1mm]
\cmidrule(l){2-3}
     & e_{\hat{3}1}=\begin{bsmallmatrix}1&&&\\&1&&\\&&-1&\\&&&-1\end{bsmallmatrix}, \
e_{\hat{1}1}=\begin{bsmallmatrix}&&1&\\&&&1\\1&&&\\&1&&\end{bsmallmatrix}, \
e_{\hat{2}1}=\begin{bsmallmatrix}&&&-1\\&&1&\\&1&&\\-1&&&\end{bsmallmatrix}, \
e_{\hat{3}2}=\begin{bsmallmatrix}&&1&\\&&&-1\\1&&&\\&-1&&\end{bsmallmatrix}
 & +1\\ \addlinespace[1mm]
     &
e_{\hat{1}2}=\begin{bsmallmatrix}-1&&&\\&1&&\\&&1&\\&&&-1\end{bsmallmatrix}, \
e_{\hat{2}2}=\begin{bsmallmatrix}&-1&&\\-1&&&\\&&&-1\\&&-1&\end{bsmallmatrix}
 & \\ \addlinespace[1mm]
\midrule
\deg 4: & e_{\hat{1}\hat{2}\hat{3}1}=\begin{bsmallmatrix}&-1&&\\1&&&\\&&&-1\\&&1&\end{bsmallmatrix},\
e_{\hat{1}\hat{2}\hat{3}2}=\begin{bsmallmatrix}&&&1\\&&1&\\&-1&&\\-1&&&\end{bsmallmatrix}
 & -1\\ \addlinespace[1mm]
\cmidrule(l){2-3}
     & e_{\hat{1}\hat{2}12}=\begin{bsmallmatrix}&&&1\\&&1&\\&1&&\\1&&&\end{bsmallmatrix}, \
e_{\hat{2}\hat{3}12}=\begin{bsmallmatrix}&-1&&\\-1&&&\\&&&1\\&&1&\end{bsmallmatrix}, \
e_{\hat{3}\hat{1}12}=\begin{bsmallmatrix}1&&&\\&-1&&\\&&1&\\&&&-1\end{bsmallmatrix}
 & +1\\ \addlinespace[1mm]
\bottomrule
\end{array}
\]
\caption{Choice of $\Cl^0(3,q)$-representations for $q=2$}\label{table.q=2}
\end{table}
With this we can find the map $\Pi\colon \Lambda^2W_2\to \mathbb{R}^{3,2}$.
\begin{theorem}
Setting $W_2\cong \mathbb{R}^4=\langle E_1,E_2,E_3,E_4\rangle$ with the $\mathfrak{spin}(3,2)$-module
structure above the map $\Pi\colon \Lambda^2W_2\to \mathbb{R}^{3,2}$ given by
\begin{gather*}
\Pi(E_1\wedge E_2)=-\hat{e}_3-e_1,\quad \Pi(E_1\wedge E_3)=-\hat{e}_2,\quad \Pi(E_1\wedge E_4)=-\hat{e}_1+e_2,\\
\Pi(E_2\wedge E_3)=\hat{e}_1+e_2,\quad \Pi(E_4\wedge E_2)=\hat{e}_2,\quad \Pi(E_3\wedge E_4)=\hat{e}_3-e_1
\end{gather*}
is $\mathfrak{spin}(3,2)$-invariant and non-degenerate.
\end{theorem}

Recall that the action of $\so(3,q)$ on the $\Cl^0(3,q)$-module $W$, and thereby $W\wedge W$, is given by the isomorphism
$\ad^{-1}\colon\so(3,q)\to\mathfrak{spin}(3,q)=\Cl^0(3,q)$, $e_i\wedge e_j\mapsto -\frac 12 e_{ij}$,
where $i,j\in\{\hat{1},\hat{2},\hat{3},1,\dots,q\}$.

\begin{proof}
Non-degeneracy is clear. It suffices to check the invariance on a generating set of $\Cl^0(3,2)$. One
such set is given by $e_{\hat{1}\hat{2}},e_{\hat{2}\hat{3}},e_{\hat{1}1},e_{12}$. Each of these map certain subspaces
of $W_2$ onto one another, hence their action on the exterior product of these subspaces vanishes. This yields the identities
\begin{align*}
-2\hat{e}_1\wedge\hat{e}_2(\hat{e}_3\pm e_1)&=0= \Pi(e_{\hat{1}\hat{2}}(E_1\wedge E_2))=\Pi(e_{\hat{1}\hat{2}}(E_3\wedge E_4)),\\
-2\hat{e}_2\wedge\hat{e}_3(\hat{e}_1\pm e_2)&=0= \Pi(e_{\hat{2}\hat{3}}(E_1\wedge E_4))=\Pi(e_{\hat{2}\hat{3}}(E_2\wedge E_3)),\\
-2\hat{e}_1\wedge e_1(\hat{e}_2)&=0= \Pi(e_{\hat{1}1}(E_1\wedge E_3))=\Pi(e_{\hat{1}1}(E_4\wedge E_2)),\\
-2e_1\wedge e_2(\hat{e}_2)&=0= \Pi(e_{12}(E_1\wedge E_3))=\Pi(e_{12}(E_4\wedge E_2)).
\end{align*}
The rest is just more computations. We start with $e_{\hat{1}\hat{2}}$:
\begin{align*}
\Pi(e_{\hat{1}\hat{2}}(E_1\wedge E_3))&=\Pi(E_2\wedge E_3)+\Pi(E_1\wedge -E_4)=\hat{e}_1+e_2+\hat{e}_1-e_2=2\hat{e}_1\\
&=-2\hat{e}_1\wedge\hat{e}_2(-\hat{e}_2)=-2\hat{e}_1\wedge\hat{e}_2(\Pi(E_1\wedge E_3)),\\
\Pi(e_{\hat{1}\hat{2}}(E_1\wedge E_4))&=\Pi(E_2\wedge E_4)+\Pi(E_1\wedge E_3)=-\hat{e}_2-\hat{e}_2=-2\hat{e}_2\\
&=-2\hat{e}_1\wedge\hat{e}_2(-\hat{e}_1+e_2)=-2\hat{e}_1\wedge\hat{e}_2(\Pi(E_1\wedge E_4)),\\
\Pi(e_{\hat{1}\hat{2}}(E_2\wedge E_3))&=\Pi(-E_1\wedge E_3)+\Pi(E_2\wedge -E_4)=\hat{e}_2+\hat{e}_2=2\hat{e}_2\\
&=-2\hat{e}_1\wedge\hat{e}_2(\hat{e}_1+e_2)=-2\hat{e}_1\wedge\hat{e}_2(\Pi(E_2\wedge E_3)),\\
\Pi(e_{\hat{1}\hat{2}}(E_4\wedge E_2))&=\Pi(E_3\wedge E_2)+\Pi(E_4\wedge -E_1)=-\hat{e}_1-e_2-\hat{e}_1+e_2=-2\hat{e}_1\\
&=-2\hat{e}_1\wedge\hat{e}_2(\hat{e}_2)=-2\hat{e}_1\wedge\hat{e}_2(\Pi(E_4\wedge E_2)).
\end{align*}
For $e_{\hat{2}\hat{3}}$:
\begin{align*}
\Pi(e_{\hat{2}\hat{3}}(E_1\wedge E_2))&=\Pi(E_4\wedge E_2)+\Pi(E_1\wedge -E_3)=\hat{e}_2+\hat{e}_2=2\hat{e}_2\\
&=-2\hat{e}_2\wedge\hat{e}_3(-\hat{e}_3-e_1)=-2\hat{e}_2\wedge\hat{e}_3(\Pi(E_1\wedge E_2)),\\
\Pi(e_{\hat{2}\hat{3}}(E_1\wedge E_3))&=\Pi(E_4\wedge E_3)+\Pi(E_1\wedge E_2)=-\hat{e}_3+e_1-\hat{e}_3-e_1=-2\hat{e}_3\\
&=-2\hat{e}_2\wedge\hat{e}_3(-\hat{e}_2)=-2\hat{e}_2\wedge\hat{e}_3(\Pi(E_1\wedge E_3)),\\
\Pi(e_{\hat{2}\hat{3}}(E_4\wedge E_2))&=\Pi(-E_1\wedge E_2)+\Pi(E_4\wedge -E_3)=\hat{e}_3+e_1+\hat{e}_3-e_1=2\hat{e}_3\\
&=-2\hat{e}_2\wedge\hat{e}_3(\hat{e}_2)=-2\hat{e}_2\wedge\hat{e}_3(\Pi(E_4\wedge E_2)),\\
\Pi(e_{\hat{2}\hat{3}}(E_2\wedge E_4))&=\Pi(E_3\wedge E_2)+\Pi(E_3\wedge -E_1)=-\hat{e}_2-\hat{e}_2=-2\hat{e}_2\\
&=-2\hat{e}_2\wedge\hat{e}_3(\hat{e}_3-e_1)=-2\hat{e}_2\wedge\hat{e}_3(\Pi(E_3\wedge E_4)).
\end{align*}
For $e_{\hat{1}1}$:
\begin{align*}
\Pi(e_{\hat{1}1}(E_1\wedge E_2))&=\Pi(E_3\wedge E_2)+\Pi(E_1\wedge E_4)=-\hat{e}_1-e_2-\hat{e}_1+e_2=-2\hat{e}_1\\
&=-2\hat{e}_1\wedge e_1(-\hat{e}_3-e_1)=-2\hat{e}_1\wedge e_1(\Pi(E_1\wedge E_2)),\\
\Pi(e_{\hat{1}1}(E_1\wedge E_4))&=\Pi(E_3\wedge E_4)+\Pi(E_1\wedge E_2)=\hat{e}_3-e_1-\hat{e}_3-e_1=-2e_1\\
&=-2\hat{e}_	1\wedge e_1(-\hat{e}_1+e_2)=-2\hat{e}_1\wedge e_1(\Pi(E_1\wedge E_4)),\\
\Pi(e_{\hat{1}1}(E_2\wedge E_3))&=\Pi(E_4\wedge E_3)+\Pi(E_2\wedge E_1)=-\hat{e}_3+e_1+\hat{e}_3+e_1=2e_1\\
&=-2\hat{e}_1\wedge e_1(\hat{e}_1+e_2)=-2\hat{e}_1\wedge e_1(\Pi(E_2\wedge E_3)),\\
\Pi(e_{\hat{1}1}(E_3\wedge E_4))&=\Pi(E_1\wedge E_4)+\Pi(E_3\wedge E_2)=-\hat{e}_1+e_2-\hat{e}_1-e_2=-2\hat{e}_1\\
&=-2\hat{e}_1\wedge e_1(\hat{e}_3-e_1)=-2\hat{e}_1\wedge e_1(\Pi(E_3\wedge E_4)).
\end{align*}
And finally for $e_{12}$:
\begin{align*}
\Pi(e_{12}(E_1\wedge E_2))&=\Pi(-E_3\wedge E_2)+\Pi(E_1\wedge E_4)=\hat{e}_1+e_2-\hat{e}_1+e_2=2e_2\\
&=-2 e_1\wedge e_2(-\hat{e}_3-e_1)=-2 e_1\wedge e_2(\Pi(E_1\wedge E_2)),\\
\Pi(e_{12}(E_1\wedge E_4))&=\Pi(-E_3\wedge E_4)+\Pi(E_1\wedge -E_2)=-\hat{e}_3+e_1+\hat{e}_3+e_1=2e_1\\
&=-2 e_1\wedge e_2(-\hat{e}_1+e_2)=-2 e_1\wedge e_2(\Pi(E_1\wedge E_4)),\\
\Pi(e_{12}(E_2\wedge E_3))&=\Pi(E_4\wedge E_3)+\Pi(E_2\wedge E_1)=-\hat{e}_3+e_1+\hat{e}_3+e_1=2e_1\\
&=-2 e_1\wedge e_2(\hat{e}_1+e_2)=-2 e_1\wedge e_2(\Pi(E_2\wedge E_3)),\\
\Pi(e_{12}(E_3\wedge E_4))&=\Pi(E_1\wedge E_4)+\Pi(E_3\wedge -E_2)=-\hat{e}_1+e_2+\hat{e}_1+e_2=2e_2\\
&=-2 e_1\wedge e_2(\hat{e}_3-e_1)=-2 e_1\wedge e_2(\Pi(E_3\wedge E_4)).\qedhere
\end{align*}
\end{proof}

To display the algebra $\mathfrak{g}$ with $q=2,\ l=1$ corresponding to $\hat{\mathcal{T}}(1)=G/H$ we
use the inclusion $\mathfrak{g}\subset\mathfrak{der}(\mathfrak{r})\subset\mathfrak{gl}(\R^{3,2}\oplus\R D\oplus W_2)$.
We give these elements as matrices with respect to the basis $\hat{e}_1,\hat{e}_2,
\hat{e}_3,e_1,e_2,D,E_1,E_2,E_3,E_4$ of $\mathfrak{r}$.\\

Recall $\mathfrak{g}=\so(3,2)\oplus\R^{3,2}\oplus\R D\oplus W_2$. We begin with $\so(3,2)$:

\scalebox{.85}{\parbox{\linewidth}{
\begin{gather*}
2\hat{e}_1\wedge\hat{e}_2=\begin{bsmallmatrix}
0&2&&&&0&&&&\\
-2&0&&&&0&&&&\\
&&0&&&0&&&&\\
&&&0&&0&&&&\\
&&&&0&0&&&&\\
0&0&0&0&0&0&0&0&0&0\\
&&&&&0&&&0&1\\
&&&&&0&&&-1&0\\
&&&&&0&0&1&&\\
&&&&&0&-1&0&&\\\end{bsmallmatrix},\quad
2\hat{e}_3\wedge\hat{e}_1=\begin{bsmallmatrix}
0&&-2&&&0&&&&\\
&0&&&&0&&&&\\
2&&0&&&0&&&&\\
&&&0&&0&&&&\\
&&&&0&0&&&&\\
0&0&0&0&0&0&0&0&0&0\\
&&&&&0&&&1&0\\
&&&&&0&&&0&1\\
&&&&&0&-1&0&&\\
&&&&&0&0&-1&&\\\end{bsmallmatrix},\\
2\hat{e}_2\wedge\hat{e}_3=\begin{bsmallmatrix}
0&&&&&0&&&&\\
&0&2&&&0&&&&\\
&-2&0&&&0&&&&\\
&&&0&&0&&&&\\
&&&&0&0&&&&\\
0&0&0&0&0&0&0&0&0&0\\
&&&&&0&&&0&1\\
&&&&&0&&&-1&0\\
&&&&&0&0&1&&\\
&&&&&0&-1&0&&\\\end{bsmallmatrix},
\end{gather*}
}}

\scalebox{.85}{\parbox{\linewidth}{
\begin{gather*}
2\hat{e}_1\wedge e_1=\begin{bsmallmatrix}
0&&&-2&&0&&&&\\
&0&&&&0&&&&\\
&&0&&&0&&&&\\
-2&&&0&&0&&&&\\
&&&&0&0&&&&\\
0&0&0&0&0&0&0&0&0&0\\
&&&&&0&&&-1&0\\
&&&&&0&&&0&-1\\
&&&&&0&-1&0&&\\
&&&&&0&0&-1&&\\\end{bsmallmatrix},\qquad
2\hat{e}_1\wedge e_2=\begin{bsmallmatrix}
0&&&&-2&0&&&&\\
&0&&&&0&&&&\\
&&0&&&0&&&&\\
&&&0&&0&&&&\\
-2&&&&0&0&&&&\\
0&0&0&0&0&0&0&0&0&0\\
&&&&&0&1&0&&\\
&&&&&0&0&-1&&\\
&&&&&0&&&-1&0\\
&&&&&0&&&0&1\\\end{bsmallmatrix},\\
2\hat{e}_2\wedge e_1=\begin{bsmallmatrix}
0&&&&&0&&&&\\
&0&&-2&&0&&&&\\
&&0&&&0&&&&\\
&-2&&0&&0&&&&\\
&&&&0&0&&&&\\
0&0&0&0&0&0&0&0&0&0\\
&&&&&0&&&0&1\\
&&&&&0&&&-1&0\\
&&&&&0&0&-1&&\\
&&&&&0&1&0&&\\\end{bsmallmatrix},\qquad
2\hat{e}_2\wedge e_2=\begin{bsmallmatrix}
0&&&&&0&&&&\\
&0&&&-2&0&&&&\\
&&0&&&0&&&&\\
&&&0&&0&&&&\\
&-2&&&0&0&&&&\\
0&0&0&0&0&0&0&0&0&0\\
&&&&&0&0&1&&\\
&&&&&0&1&0&&\\
&&&&&0&&&0&1\\
&&&&&0&&&1&0\\\end{bsmallmatrix},\\
2\hat{e}_3\wedge e_1=\begin{bsmallmatrix}
0&&&&&0&&&&\\
&0&&&&0&&&&\\
&&0&-2&&0&&&&\\
&&-2&0&&0&&&&\\
&&&&0&0&&&&\\
0&0&0&0&0&0&0&0&0&0\\
&&&&&0&-1&0&&\\
&&&&&0&0&-1&&\\
&&&&&0&&&1&0\\
&&&&&0&&&0&1\\\end{bsmallmatrix},\qquad
2\hat{e}_3\wedge e_2=\begin{bsmallmatrix}
0&&&&&0&&&&\\
&0&&&&0&&&&\\
&&0&&-2&0&&&&\\
&&&0&&0&&&&\\
&&-2&&0&0&&&&\\
0&0&0&0&0&0&0&0&0&0\\
&&&&&0&&&-1&0\\
&&&&&0&&&0&1\\
&&&&&0&-1&0&&\\
&&&&&0&0&1&&\\\end{bsmallmatrix},
\end{gather*}
}}

\scalebox{.85}{\parbox{\linewidth}{
\begin{gather*}
2e_1\wedge e_2=\begin{bsmallmatrix}
0&&&&&0&&&&\\
&0&&&&0&&&&\\
&&0&&&0&&&&\\
&&&0&-2&0&&&&\\
&&&2&0&0&&&&\\
0&0&0&0&0&0&0&0&0&0\\
&&&&&0&&&-1&0\\
&&&&&0&&&0&1\\
&&&&&0&1&0&&\\
&&&&&0&0&-1&&\\\end{bsmallmatrix}.
\end{gather*}
}}\\
Recall that $2e_1\wedge e_2$ generates the isotropy algebra $\mathfrak{h}$.

Next we describe the element $2D$:\\
\scalebox{.85}{\parbox{\linewidth}{
\[
2D=\begin{bsmallmatrix}
2&&&&&0&&&&\\
&2&&&&0&&&&\\
&&2&&&0&&&&\\
&&&2&&0&&&&\\
&&&&2&0&&&&\\
0&0&0&0&0&0&0&0&0&0\\
&&&&&0&1&&&\\
&&&&&0&&1&&\\
&&&&&0&&&1&\\
&&&&&0&&&&1\\\end{bsmallmatrix}.
\]
}}

The generators of $V=\R^{3,2}$ are:

\scalebox{.95}{\parbox{\linewidth}{
\begin{gather*}
\hat{e}_1=\begin{bsmallmatrix}
&&&&&-1&&&&\\
&&&&&0&&&&\\
&&&&&0&&&&\\
&&&&&0&&&&\\
&&&&&0&&&&\\
0&0&0&0&0&0&0&0&0&0\\
&&&&&0&&&&\\
&&&&&0&&&&\\
&&&&&0&&&&\\
&&&&&0&&&&\\\end{bsmallmatrix},\qquad
\hat{e}_2=\begin{bsmallmatrix}
&&&&&0&&&&\\
&&&&&-1&&&&\\
&&&&&0&&&&\\
&&&&&0&&&&\\
&&&&&0&&&&\\
0&0&0&0&0&0&0&0&0&0\\
&&&&&0&&&&\\
&&&&&0&&&&\\
&&&&&0&&&&\\
&&&&&0&&&&\\\end{bsmallmatrix},\qquad
\hat{e}_3=\begin{bsmallmatrix}
&&&&&0&&&&\\
&&&&&0&&&&\\
&&&&&-1&&&&\\
&&&&&0&&&&\\
&&&&&0&&&&\\
0&0&0&0&0&0&0&0&0&0\\
&&&&&0&&&&\\
&&&&&0&&&&\\
&&&&&0&&&&\\
&&&&&0&&&&\\\end{bsmallmatrix},\\
e_1=\begin{bsmallmatrix}
&&&&&0&&&&\\
&&&&&0&&&&\\
&&&&&0&&&&\\
&&&&&-1&&&&\\
&&&&&0&&&&\\
0&0&0&0&0&0&0&0&0&0\\
&&&&&0&&&&\\
&&&&&0&&&&\\
&&&&&0&&&&\\
&&&&&0&&&&\\\end{bsmallmatrix},\qquad
e_2=\begin{bsmallmatrix}
&&&&&0&&&&\\
&&&&&0&&&&\\
&&&&&0&&&&\\
&&&&&0&&&&\\
&&&&&-1&&&&\\
0&0&0&0&0&0&0&0&0&0\\
&&&&&0&&&&\\
&&&&&0&&&&\\
&&&&&0&&&&\\
&&&&&0&&&&\\\end{bsmallmatrix}.
\end{gather*}
}}

\pagebreak
Finally we have the $4$ basis elements of $W_2$:\\
\scalebox{.95}{\parbox{\linewidth}{
\begin{align*}
2E_1=\begin{bsmallmatrix}
&&&&&0&&&&-2\\
&&&&&0&&&-2&\\
&&&&&0&&-2&&\\
&&&&&0&&2&&\\
&&&&&0&&&&2\\
0&0&0&0&0&0&0&0&0&0\\
&&&&&-1&&&&\\
&&&&&0&&&&\\
&&&&&0&&&&\\
&&&&&0&&&&\\\end{bsmallmatrix},\quad &
2E_2=\begin{bsmallmatrix}
&&&&&0&&&2&\\
&&&&&0&&&&-2\\
&&&&&0&2&&&\\
&&&&&0&2&&&\\
&&&&&0&&&2&\\
0&0&0&0&0&0&0&0&0&0\\
&&&&&0&&&&\\
&&&&&-1&&&&\\
&&&&&0&&&&\\
&&&&&&&&&\\\end{bsmallmatrix}\\
2E_3=\begin{bsmallmatrix}
&&&&&0&&-2&&\\
&&&&&0&2&&&\\
&&&&&0&&&&2\\
&&&&&0&&&&-2\\
&&&&&0&&-2&&\\
0&0&0&0&0&0&0&0&0&0\\
&&&&&0&&&&\\
&&&&&0&&&&\\
&&&&&-1&&&&\\
&&&&&0&&&&\\\end{bsmallmatrix},\quad &
2E_4=\begin{bsmallmatrix}
&&&&&0&2&&&\\
&&&&&0&&2&&\\
&&&&&0&&&-2&\\
&&&&&0&&&2&\\
&&&&&0&-2&&&\\
0&0&0&0&0&0&0&0&0&0\\
&&&&&0&&&&\\
&&&&&0&&&&\\
&&&&&0&&&&\\
&&&&&-1&&&&\\\end{bsmallmatrix}
\end{align*}
}} 
\section{Nomizu maps of homogeneous $3$-$(\alpha,\delta)$-Sasaki manifolds}
\subsection{Nomizu map of the canonical connection}
By the Nomizu theorem invariant connections on reductive homogeneous spaces $M=G/H$ are in bijective correspondence with isotropy equivariant maps $\Lambda\colon \mathfrak{m}\times\mathfrak{m}\to\mathfrak{m}$, where $\mathfrak{m}$ is a reductive complement to the isotropy algebra $\mathfrak{h}\subset\mathfrak{g}$. For fundamental vector fields $X,Y$ considered as elements in $\mathfrak{m}$ this correspondence is given by $\Lambda^\nabla_XY=\nabla_{X_0}Y-[X,Y]_0$, compare \cite[Corollary 2.2, p.191]{KobNom}.
By \cite[Proposition X.2.3, p. 191]{KobNom} the torsion $T^\nabla$ of the connection corresponding to $\Lambda^\nabla$ is given by
\begin{equation}\label{lambdatorsion}
T^\nabla(X,Y)= \Lambda^\nabla_X Y-\Lambda^\nabla_Y X-[X,Y]_{\mathfrak{m}}.
\end{equation}

For the following theorem we need a similar statement to \autoref{VComm} for two fundamental vector fields. Note that even in the case when $X,Y\in\mathfrak{g}$ are horizontal in the origin they fail to be horizontal in other points. Yet we have
\begin{lem}\label{VCommFundamental}
Let $G\curvearrowright M$ act by automorphisms on a $3$-$(\alpha,\delta)$-Sasaki manifold $M$ from the left, $X,Y\in\mathfrak{g}$ with fundamental vector fields $\hat{X},\hat{Y}\in\mathfrak{X}(M)$. Then
\[
\widehat{[X,Y]}_{\mathcal{V}}=-\sum_{i=1}^3\big(2\alpha\Phi_i(\hat{X},\hat{Y})+2(\alpha-\delta)\eta_j\wedge\eta_k(\hat{X},\hat{Y})\big)\xi_i.
\]
\end{lem}
\begin{proof}
Since $G$ is a group of automorphisms we have $L_{\hat{X}}\eta_i=0$ so
\begin{align*}
d\eta_i(\hat{X},\hat{Y}) &=\hat{X}(\eta_i(\hat{Y}))-\hat{Y}(\eta_i(\hat{X}))-\eta_i([\hat{X},\hat{Y}])\\
&=\eta_i([\hat{X},\hat{Y}])-\eta_i([\hat{Y},\hat{X}])-\eta_i([\hat{X},\hat{Y}])=\eta_i([\hat{X},\hat{Y}])=-\eta_i(\widehat{[X,Y]}).
\end{align*}
Using $\widehat{[X,Y]}_\mathcal{V}=\sum_{i=1}^3\eta_i(\widehat{[X,Y]})\xi_i$ and \eqref{differential_eta} yields the result.
\end{proof}

\begin{theorem}
The Nomizu map for the canonical connection $\Lambda^\nabla\colon \mathfrak{m}\times\mathfrak{m}\to\mathfrak{m}$ of a homogeneous $3$-$(\alpha,\delta)$-Sasaki manifold with $\mathfrak{m}=\mathcal{V}\oplus\mathcal{H}$ is given by
\[
\Lambda^\nabla_X Y=\begin{cases}
\Lambda^{g_N}_X Y & X,Y\in\mathcal{H}\\
\frac{\beta}{2\delta}[X,Y] & X,Y\in\mathcal{V}\\
[X,Y]-2\alpha\sum_{i=1}^3\eta_i(X)\varphi_i Y & X\in\mathcal{V},Y\in\mathcal{H}\\
0 & X\in\mathcal{H},Y\in\mathcal{V},\end{cases}
\]
where $\Lambda^{g_N}\colon\mathcal{H}\times\mathcal{H}\to\mathcal{H}$ is the Nomizu map of the Levi-Civita connection on the homogeneous base of the canonical submersion.
\end{theorem}

\begin{proof}
We first prove that the torsion of $\Lambda^\nabla$ given by \eqref{lambdatorsion} agrees with the canonical torsion
\[
T^\nabla=2\alpha\sum_{i=1}^3 \eta_i\wedge\Phi_i^\mathcal{H}+2(\delta-4\alpha)\eta_{123}.
\]
We begin with the case $X,Y\in \mathcal{H}$. Evaluating \autoref{VCommFundamental} in the origin $T_0M\cong\mathfrak{m}$ we obtain $[X,Y]_\mathcal{V}=-2\alpha\sum_{i=1}^3\Phi_i(X,Y)\xi_i$. Thus,
\begin{align*}
\Lambda^\nabla_X Y-\Lambda^\nabla_YX-[X,Y]&=\Lambda^{g_N}_X Y-\Lambda^{g_N}_Y X-[X,Y]_\mathcal{H}-[X,Y]_\mathcal{V}\\
&=T^{g_N}(X,Y)+2\alpha\sum_{i=1}^3\Phi_i(X,Y)\xi_i=0+2\alpha\sum_{i=1}^3\Phi_i(X,Y)\xi_i=T^\nabla(X,Y)
\end{align*}
as the torsion $T^{g_N}=0$ of the Levi-Civita connection on the base vanishes.
Suppose now that $X=\xi_i$ and $Y\in\mathcal{H}$. Then
\[
\Lambda^\nabla_X Y-\Lambda^\nabla_YX-[X,Y]=[X,Y]-2\alpha\varphi_i Y-[X,Y]=-2\alpha\varphi_i Y=T^\nabla(X,Y).
\]
Finally if both $X=\xi_i,Y=\xi_j\in\mathcal{V}$ and $(ijk)$ an even permutation of $(123)$ we have
\begin{align*}
\Lambda^\nabla_X Y-\Lambda^\nabla_YX-[X,Y]&=\frac{\beta}{2\delta}[X,Y]-\frac{\beta}{2\delta}[Y,X]-[X,Y]=(2-\frac{4\alpha}{\delta}-1)[\xi_i,\xi_j]\\
&=(1-\frac{4\alpha}{\delta})2\delta\xi_k=T^\nabla(X,Y).
\end{align*}
We further need to verify that $\Lambda^\nabla_X\in\mathfrak{so}(\mathfrak{m})$ for all $X\in \mathfrak{m}$, that is $g(\Lambda^\nabla_XY,Z)+g(Y,\Lambda_X^\nabla Z)=0$ for all $Y,Z\in\mathfrak{m}$. Suppose $X\in\mathcal{H}$ and $Y,Z\in\mathcal{H}$. Then $\Lambda^\nabla_X=\Lambda^{g_N}_X$ and $g|_\mathcal{H}=g_N$ thus $\Lambda^\nabla_X$ is metric as $\Lambda^{g_N}_X$ is. If $Y\in\mathcal{V}$ and $Z\in\mathcal{H}$ we find that $\Lambda^\nabla_XY=0$ by definition while $\Lambda^\nabla_X Z\in\mathcal{H}$ is orthogonal to $Y$. Analogously, if $Y,Z\in\mathcal{V}$ the Nomizu map $\Lambda^\nabla_X$ acts trivially on both sides.
Now suppose $X\in\mathcal{V}$. Then by \cite[Corollary 2.3.1]{AgrDil} $X$ is Killing as a linear combination of the $\xi_i$ and thus the Lie derivative $L_X\in\mathfrak{so}(\mathfrak{m})$. In particular, if $Y,Z\in\mathcal{V}$ then $\Lambda^\nabla_X Y=\frac{\beta}{2\delta}L_X Y$ is metric. If $Y\in\mathcal{H}$ we have
\[
g(\Lambda^\nabla_XY,Z)=g(L_X Y-2\alpha\sum\eta_i(X)\varphi_iY,Z)=-g(Y,L_X Z-2\alpha\sum\eta_i(X)\varphi_i Z).
\]
If $Z\in \mathcal{H}$ the right hand side is just $-g(Y,\Lambda^\nabla_X Z)$ by definition while for $Z\in\mathcal{V}$
\[
L_X Z-2\alpha\sum\eta_i(X)\varphi_i Z\in\mathcal{V}
\]
is perpendicular to $Y$. Hence, $g(\Lambda^\nabla_X Y,Z)=0=-g(Y,\Lambda^\nabla_X Z)$.
\end{proof}
\begin{rem}
For symmetric spaces the Levi-Civita connection corresponds to the trivial Nomizu map $\Lambda^{g_N}=0$. The Nomizu map of the Alekseevsky base is given in \cite[Lemma 5, p. 35]{CortesNewConstr}.
\end{rem}
%
\subsection{Nomizu maps in the symmetric base case}
In the case of a positive homogeneous $3$-$(\alpha,\delta)$-Sasaki manifold or its non-compact sibling the Nomizu map $\Lambda^\nabla$ simplifies drastically.
\begin{prop}\label{lemma_canonical}
The canonical connection $\nabla$ of a homogeneous $3$-$(\alpha,\delta)$-Sasaki manifold over a Wolf space or its non-compact dual corresponds to the map
\[
\Lambda^\nabla_X=\begin{cases}
0, & X\in \mathcal{H}\\
\frac{\beta}{2\delta}\ad X, & X\in\mathcal{V}.
\end{cases}\]
\end{prop}
\begin{proof}
In the case of a Riemannian symmetric space the Levi--Civita connection agrees with the Ambrose--Singer connection. Thus, $\Lambda^{g_N}\equiv 0$. Now let $X\in\mathcal{V}$ and $Y\in\mathcal{H}$. Then
\[
2\alpha\sum_{i=1}^3 \eta_i(X) \varphi_i|_{\mathcal{H}}=\frac {2\alpha}\delta\sum_{i=1}^3 \eta_i(X) \ad{\xi_i}=\frac {2\alpha}\delta \ad X.
\]
It follows $\Lambda^\nabla_XY=(1-\frac{2\alpha}{\delta})[X,Y]=\frac{\beta}{2\delta}[X,Y]$.
\end{proof}
\begin{rem}
As noted in \autoref{Rem_Constr} the homogeneous $3$-$(\alpha,\delta)$-Sasaki space is naturally reductive if and only if $\beta=0$. In this case the Ambrose-Singer connection is metric. In fact, \autoref{lemma_canonical} shows that in this case the canonical and Ambrose-Singer connections agree.
\end{rem}
\begin{prop}
The Levi--Civita connection corresponds to the map
\[
\Lambda^g_X Y= \begin{cases}
\frac 12 [X,Y]_{\mathfrak{m}} & X,Y\in\mathcal{V} \text{ or } X,Y\in\mathcal{H}\\
(1-\frac\alpha\delta)[X,Y] & X\in\mathcal{V},Y\in\mathcal{H}\\
\frac \alpha\delta[X,Y] & X\in\mathcal{H},Y\in\mathcal{V}.\end{cases}
\]
\end{prop}
\begin{proof}
Note that the correspondence is $\Lambda^\nabla_X Y=\nabla_X Y-[X,Y]$, for $X,Y\in\mathfrak{m}$, and the canonical connection is given by $\nabla=\nabla^g +\frac 12 T$ where the canonical torsion is given by \eqref{torsion01}, or equivalently \eqref{torsion02}.
Thus we have $\Lambda^g_X Y=\nabla^g_X Y-[X,Y]=\nabla_X Y-[X,Y] -\frac 12 T(X,Y)=\Lambda^\nabla_X Y-\frac 12 T(X,Y)$. Again we look at each case individually. Let $X,Y\in \mathcal{H}$. Then
\begin{align*}
\frac 12 T(X,Y)&=\alpha\sum_{i=1}^3 \Phi_i(X,Y)\xi_i=\alpha \sum_{i=1}^3g(X,\varphi_i Y)\xi_i\\
&=\frac \alpha\delta \sum_{i=1}^3 g(X,[\xi_i,Y])\xi_i=\frac{\alpha}{8\alpha\delta^2(n+2)}\sum_{i=1}^3\kappa(X,[Y,\xi_i])\xi_i\\
&= \frac{1}{8\delta^2(n+2)}\sum_{i=1}^3\kappa([X,Y],\xi_i)\xi_i\\
&=-\frac 12 \sum_{i=1}^3 g([X,Y]_\mathfrak{m},\xi_i)\xi_i=-\frac 12 [X,Y]_\mathfrak{m},
\end{align*}
where we have used that $\kappa(\mathfrak{h},\mathcal{V})=0$ and $[X,Y]_\mathfrak{m}\in\mathcal{V}=\mathrm{span}\{\xi_1,\xi_2,\xi_3\}$. Thus
\[
\Lambda^g_X Y=\Lambda^\nabla_X Y -\frac 12 T(X,Y)=\frac 12 [X,Y]_\mathfrak{m}
\]
For vertical vectors $X=\xi_i,Y=\xi_j$, $(ijk)$ an even permutation of $(123)$, we find
\begin{align*}
\Lambda^g_{\xi_i} \xi_j&=\Lambda^\nabla_{\xi_i} \xi_j-(\delta-4\alpha)\eta_{123}(\xi_i,\xi_j,\cdot)=\frac {\beta}{2\delta} [\xi_i,\xi_j]- (\delta-4\alpha)\xi_k\\
&=\beta\xi_k+(\delta-4\alpha)\xi_k=\delta\xi_k=[\xi_i,\xi_j]
\end{align*}
and by linearity for arbitrary $X,Y\in\mathcal{V}$. Let $X\in\mathcal{V},Y\in\mathcal{H}$ then
\begin{align*}
T(X,Y)&=2\alpha\sum_{i=1}^3\eta_i\wedge\Phi_i^{\mathcal{H}}(X,Y,\cdot\,)=2\alpha\sum_{i=1}^3\eta_i(X)\Phi_i^{\mathcal{H}}(Y,\cdot\,)\\
&=-2\alpha\sum_{i=1}^3\eta_i(X)\varphi_iY=-\frac {2\alpha}{\delta}\left[\sum_{i=1}^3\eta_i(X)\xi_i,Y\right]=-\frac{2\alpha}\delta [X,Y]
\end{align*}
and thus
\begin{align*}
\Lambda^g_X Y&=\Lambda^\nabla_X Y-\frac 12 T(X,Y)=(1-\frac {2\alpha}{\delta})[X,Y]+\frac\alpha\delta[X,Y]=(1-\frac{\alpha}{\delta})[X,Y].
\end{align*}
For the final expression $X\in\mathcal{H},Y\in\mathcal{V}$ use the above identity for $T$ with $X\leftrightarrow Y$. Then
\begin{align*}
\Lambda^g_X Y=\Lambda^\nabla_X Y-\frac 12 T(X,Y)&=\frac 12 T(Y,X)=\frac\alpha\delta [X,Y]. \qedhere
\end{align*}
\end{proof}

The Nomizu maps allow for a detailed investigation of homogeneous $3$-$(\alpha,\delta)$-Sasaki manifolds. A thorough discussion of curvature operators and properties will be carried out in an upcoming publication \cite{ADS21}.

%
%

\subsection*{Declarations}
No funding was received to assist with the preparation of this manuscript.
The authors have no relevant financial or non-financial interests to disclose.
Data sharing is not applicable to this article as no datasets were generated or analysed during the current study.


%
\noindent
Ilka Agricola and Leander Stecker, Fachbereich Mathematik und Informatik, 
Philipps-Universit\"at Marburg,
Campus Lahnberge, 35032 Marburg, Germany. \texttt{agricola@mathematik.uni-marburg.de}, \texttt{stecker@mathematik.uni-marburg.de}

\bigskip\noindent
Giulia Dileo, Dipartimento di Matematica, Universit\`a degli Studi di Bari Aldo Moro,
Via E. Orabona 4, 70125 Bari, Italy.
\texttt{giulia.dileo@uniba.it}

\end{document}